\documentclass[12pt]{amsart}
\usepackage{amsmath, amssymb}
\usepackage{mathtools, a4wide, mathrsfs, bm}
\usepackage{paralist}
\usepackage{xcolor}
\usepackage[T1]{fontenc}

\usepackage[colorlinks=true,linkcolor=blue,citecolor=blue,pdfpagelabels=false]{hyperref}
\numberwithin{equation}{section}

\theoremstyle{plain}
\newtheorem{thm}{Theorem}[section]
\newtheorem{lem}[thm]{Lemma}
\newtheorem{prop}[thm]{Proposition}
\newtheorem{cor}[thm]{Corollary}
\newcommand{\thmref}[1]{Theorem~\ref{#1}}
\newcommand{\lemref}[1]{Lemma~\ref{#1}}
\newcommand{\propref}[1]{Proposition~\ref{#1}}
\newcommand{\corref}[1]{Corollary~\ref{#1}}

\theoremstyle{definition}
\newtheorem{rmk}[thm]{Remark}
\newtheorem{defi}[thm]{Definition}

\newcommand{\rmkref}[1]{Remark~\ref{#1}}

\newcommand{\mbb}{\mathbb}
\newcommand{\x}{\textbf}
\newcommand{\mf}{\mathbf}
\newcommand{\q}{\quad}
\newcommand{\qq}{\qquad}

\newcommand{\mc}{\mathcal}
\newcommand{\mk}{\mathfrak}
\newcommand{\mrm}{\mathrm}
\newcommand{\sltwo}{\mrm{SL}(2, \mf Z)}
\newcommand{\spn}{\mrm{Sp}(n,\mf Z)}
\newcommand{\spth}{\mrm{Sp}(3,\mf Z)}

\newcommand{\spt}{\mrm{Sp}(3, \mf Z)}

\newcommand{\sltwor}{\mrm{SL}(2, \mf R)}

\newcommand{\spne}{\mrm{Sp}(n-1,{\mf R})}

\newcommand{\spnr}{\mrm{Sp}(n, \mf R)}
\newcommand{\glnz}{\mrm{GL}(n, \mf Z)}
\newcommand{\glnc}{\mrm{GL}(n, \mf C)}
\newcommand{\slnz}{\mrm{SL}(n, \mf Z)}

\newcommand{\ut}{\underset}

\newcommand{\0}{\underline{0}}

\newcommand{\z}{\mathfrak z}
\newcommand{\mN}{M_{\kappa}(\Gamma(N'))}

\newcommand{\wm}{\widetilde{M}}
\newcommand{\wnu}{\widetilde{\nu}}
\newcommand{\wmu}{\widetilde{\mu}}
\newcommand{\n}{\nonumber}
\newcommand{\te}{\textbullet \, }

\newcommand{\tr}{\mathrm{tr}}

\newcommand{\dis}{\mrm{disc}}

\newcommand*{\QEDB}{\hfill\ensuremath{\square}}

\begin{document}

\title[Fundamental Fourier coefficients]{On fundamental Fourier coefficients of Siegel modular forms}

\author{Siegfried B\"ocherer}
\address{Institut f\"ur Mathematik\\
Universit\"at Mannheim\\
68131 Mannheim (Germany).}
\email{boecherer@math.uni-mannheim.de}

\author{Soumya Das}
\address{Department of Mathematics\\ 
Indian Institute of Science\\ 
Bangalore -- 560012, India \\
and Humboldt Fellow, Universit\"{a}t Mannheim\\
68131 Mannheim (Germany).}
\email{soumya@iisc.ac.in, sdas@mail.uni-mannheim.de}

\date{}
\subjclass[2000]{Primary 11F30, 11F46, Secondary 11F50} 
\keywords{Fourier coefficients, Siegel modular forms, fundamental discriminant, vector-valued, nonvanishing}

\begin{abstract}
We prove that if $F$ is a non-zero (possibly non-cuspidal) vector-valued Siegel modular form of any degree, then it has infinitely many non-zero Fourier coefficients which are indexed by half-integral matrices having odd, square-free (and thus fundamental) discriminant. The proof uses an induction argument in the setting of vector-valued modular forms. Further, as an application of a variant of our result and complementing the work of A. Pollack, we show how to obtain an unconditional proof of the functional equation of the spinor $L$-function of a holomorphic cuspidal Siegel eigenform of degree $3$ and level $1$.
\end{abstract}

\maketitle 

\section{Introduction}
Fourier coefficients of Siegel modular forms have been objects of continued interest over the years. It is a very useful fact, not only theoretically, but also in computation, that such a form $F \in M^n_k$,  where $M^n_k$ (resp. $S^n_k$) being the space of (resp. cuspidal) Siegel modular forms on $\spn$ of scalar weight $k$, is determined by finitely many of its Fourier coefficients $a_F(T)$, described uniformly in terms of the weight $k$ and $n$. Here and henceforth, for $F \in M^n_k$, we write the Fourier expansion of $F$ as
\[ F(Z) = \sum_{T \in \Lambda_n} a_F(T) e(TZ),  \]
and $\Lambda_n$ is the set of all half-integral positive semi-definite matrices (see section~\ref{prelim}) of size $n$ and $e(TZ):= \exp (2 \pi i \tr TZ)$. Such a result is known in the literature as `Sturm'-bound, see \cite{Fr, Kl}, even though this result was known from the works of Hecke or Maa{\ss}. For a more recent version of this in the context of the determination of elliptic cusp forms by `square-free' Fourier coefficients, see \cite{AD}.

Equally important are results concerning the determination of a Siegel modular form $F$ by Fourier coefficients $a_F(T)$ supported on $T \in \Lambda_n^+$ (which are positive-definite members of $\Lambda_n$) with the discriminant of $2T$ (which we denote by $\dis(2T)$) varying in an arithmetically interesting subset $\mc{S}$ of natural numbers, e.g., square-free numbers or fundamental discriminants. For the many works along this line of research, we refer the reader to the introductions in \cite{AD, saha}. Let us just note here that in the case of newforms of half-integral weights (via Waldspurger's formula) such a result is equivalent to the determination of these forms by the twisted $L$-functions of their Shimura lifts.

When the set $\mc{S}$ as above consists of all fundamental discriminants, A. Saha \cite{saha} proved an affirmative result on $S^2_k$; this result has applications to the representation theory of automorphic forms, see the discussion in \cite[Introduction]{saha}. It is of course desirable to generalise the results of \cite{saha} to higher degree Siegel cusp forms (including vector-valued modular forms), and also to include the space of Eisenstein series. In fact, these aspects were mentioned as `\textit{difficult open}' problems in \cite[remark~2.6,~2.7]{saha}. While the latter (in degree $2$) was addressed to in \cite[Prop.~7.7]{BD2} by the authors, in this paper, we settle both of these questions in the most general case (for full level), in particular including vector-valued modular forms. We note that one of the most natural settings for the problem at hand (also noted in \cite[remark~2.7]{saha}), is to consider the set $\mc{S}$ to be all those lattices $2T$ which are maximal in the set of even integral lattices of a given rank. It is this viewpoint that we would consider in this paper. 

Let us now state our main result. Let $\rho$ be a polynomial, not necessarily irreducible representation of $ \mrm{GL}(n,\mf C)$. Denote by $M^n_\rho$ the vector space of holomorphic vector-valued Siegel modular forms on $\spn$ with automorphy factor $\rho$ (see section~\ref{prelim} for more details) with `determinantal' weight $k(\rho) $ (see section~\ref{prelim} for the definition). We need one more piece of notation. Let $M \in \Lambda_n^+$ and denote by $d_M$ its `absolute discriminant' (i.e., ignoring the usual sign), defined by 
\begin{equation*}
 d_M := |\mrm{disc}(2M)|= \begin{cases}    \det(2M) \text{ if } n \text{ is even,} \\
\frac{1}{2} \det(2M) \text{ if } n \text{ is odd}. \end{cases}
\end{equation*}
Further for $X \ge 1$, put
\[  {\mk S}_F (X) := \{ d \leq X, \,  d \, \mbox{\rm{odd, square-free}} \,|  d_T=d \, \mbox{\rm{ for some }} T \mbox{\rm{ and }}  a_F(T) \neq 0 \}.  \]

Define the function $\varrho(n)$ by
\begin{equation}
\varrho(n) = \begin{cases}  3/2 &\text{ if } n \text{ is even}, \\
1 &\text{ if } n \text{ is odd}.
  \end{cases}
\end{equation}
\begin{thm} \label{mainthm}
Let $F \in M^n_\rho$ be non-zero and $k(\rho) - \frac{n}{2} \ge \varrho(n) $. When $n$ is even, assume that $F$ is cuspidal. Then there exist infinitely many $\glnz$-inequivalent $T \in \Lambda_n^+$ such that $d_T$ is odd and square-free, and $a_F(T)\neq 0$. Moreover, the following stronger quantitative result holds: for any given $\epsilon>0$,
\begin{align}
\# {\mk S}_F (X)  \gg\begin{cases}   X \cdot (\log X)^{-1/2}  &\text{ if } n \text{ is odd} ,  \\ 
  X^{5/8 - \epsilon} &\text{ if } n \text{ is even and } F \text{ is cuspidal,}  \\
  X &\text{ if } F \text{ is scalar-valued, non-cuspidal of weight } k, \\
   & \q k \text{ is even, } k>n+1 \text{ and } n \text{ is odd. }  \\
\end{cases}
\end{align}
Here the implied constant depends only on $F$ and $\epsilon$.
\end{thm}

Here we call $S,T \in \Lambda_n^+$ equivalent under $\glnz$ if therre exists $U \in \glnz$ such that $T=U^t S U$. 
For a version of \thmref{mainthm} where we count prime discriminants, see \thmref{oddthm}.

In particular, taking $\rho = \det^k$ (for $k - \frac{n}{2} \ge \varrho(n) $) \thmref{mainthm} applies to scalar-valued Siegel modular forms of weight $k$. For more information about the lower bound on the quantity $k(\rho)$, see \rmkref{wts}. Since the $T$ appearing in \thmref{mainthm} arise from maximal (even) lattices, the statement of \thmref{mainthm} also holds a fortiori for maximal lattices.

We add here that the different lower bounds for the quantity $\# \mk S_F(X)$ in \thmref{mainthm} depending on the parity of the degree $n$ occur due to our different treatment
of these cases. The first and last lower bounds $\# \mk S_F(X)$ emanate from an argument involving multiplicity-one for integral weights whereas the second lower bound relies on the existence of unconditional bounds on Fourier coefficients of  half-integral elliptic cusp forms. Let us mention that if $n$ is odd, we encounter integral weights, and half-integral weights otherwise. The reader may note that when $n$ is even, we do not have a result on non-cusp forms. This is due to some complications arising from half-integral weights.

Our proof uses induction on the degree $n$, with the 
Fourier-Jacobi expansion as a main tool. 
The proof clearly decomposes into a preparatory part (called \textit{part~A}) of algebraic and number theoretic
considerations and an analytic part (called \textit{part~B}), 
where non-vanishing properties of Fourier coefficients for elliptic modular forms of half-integral or integral weights via some version of the Rankin-Selberg method play an essential role. 

\subsection*{Let us now briefly explain the steps of \textit{part~A}}
The main aim of this part is to reduce 
the question to a problem on certain elliptic modular forms. The results in this part should hold more generally over the classical tube domains I-IV (as in e.g. \cite{Yam}), but we do not pursue it here mainly because such a treatment may obscure the technical points of the paper. We may return to this point in a future work.

\subsubsection*{Step~1} This step assures the existence of a non-vanishing Jacobi coefficient $\varphi_T= \varphi_T(\tau, \mk z)$ ($T \in \Lambda_{n-1}^+, \tau \in \mf H, \mk z \in \mf C^{(1,n-1)}$)
of $F$ with discriminant of $T$ being odd and square-free.
To prove this, we consider the Taylor expansion of $F$ with respect to $\mk z$ around the origin. Then the non-vanishing Taylor coefficients of the lowest homogeneous degree
give rise to a possibly vector-valued modular form of degree $n-1$. By induction, this modular form of degree $n-1$ has a nonvanishing Fourier coefficient indexed by $T\in \Lambda^+_{n-1} $ with $d_T$ being odd and square-free.
For the original modular form $F$ of degree $n$ this assures the existence of a nonvanishing $\varphi_T$ with $T$ as above.
We mention here that in order to make the induction work, we have to deal with vector-valued modular forms from the beginning, even if we started from scalar-valued modular forms (see \propref{vec}).

\subsubsection*{Step~2} This step assures that the (possibly vector-valued) non-vanishing Fourier-Jacobi coefficient $\varphi_T$ from \textit{Step~1} has a non-vanishing component $\varphi_T^{(r)}$ which is actually
a scalar valued Jacobi form. It is this Jacobi form which we will focus on throughout the rest of our proof.

\subsubsection*{Step~3} This step is concerned with  the theta expansion of $\varphi_T^{(r)}$: we show that there exist at least one non-zero theta-component $h_\mu= h_\mu(\tau)$ ($\mu \in \mf Z^{n-1} / (2T) \mf Z^{n-1}$) of $\varphi_T^{(r)}$ such that $T^{-1}[\mu/2]$ has the highest possible denominator (essentially equal to $d_T$); we call such $\mu$ `primitive'. This result (\propref{primh}) is of independent interest, and is an intrinsic result in the theory of Jacobi forms. We finish \textit{Step~3} by setting up the desired non-vanishing properties for the Fourier coefficients of such $h_{\mu}$, which follows from \textit{part~B} (discussed below) to prove our theorem. Here we encounter both integral and half-integral weights according to the parity of $n \bmod 2$. We also have to take special care of the prime $p=2$ during the induction step (while passing from odd to even degrees), and ensure we do not get unnecessary high powers of $2$, see subsection~\ref{2adjust}.

\subsection*{Let us now briefly explain the steps of \textit{part~B}}

Our induction steps in \textit{Part~A} do not `see' whether the modular form is cuspidal or not; however in this part, such a distinction becomes prominent. Moreover let us note that \textit{Part~B} actually serves two purposes:

(i) first, it covers the base case $n=1$ of the induction procedure; i.e., proves \thmref{mainthm} when $f \in M^1_k$ with $k \in \mf N$. If $f \in S^1_k$, then such a result is already known from \cite[Thm.~6 and Prop.~5.8]{AD}. The extension to $M^1_k$ is trivial.

(ii) second, and more importantly, it helps to glue the non-zero `square-free' Fourier coefficients of $h_\mu$ (see section~\ref{2adjust}) with $T \in \Lambda_{n-1}^+$ to obtain some $\mc T \in \Lambda_n^+$ which is also `square-free'. The treatment of these $h_\mu$'s however leads us to both integral and half integral weights over the principal congruence subgroups. Thus in the discussion below, we focus only on (ii). We give two approaches for the analytic part, and call them Method~1 (see section~\ref{ell-icusp},~\ref{ell-hcusp}) and Method~2 (see section~\ref{ell-ncusp}). We feel that both the methods have their own advantages and limitations. These are discussed below.

\subsubsection*{Step~1} The analytic part first analyses the Fourier expansion of the degree one cusp forms $h_{\mu}$ for primitive $\mu$ from \textit{Step~3} of \textit{part~A} discussed above. Here some analytic number theory of modular forms comes in, and we essentially adapt an argument from \cite{AD, saha} using either a classical Rankin-Selberg method or a `smoothed' version (Method~1) in the case of half-integral weights, to the groups $\Gamma_1(N)$. The primitiveness of $\mu$ is crucial here. This method has the advantage that it holds uniformly for both integral and half-integral weight cusp forms, and does not depend on any multiplicity-one result. See \thmref{partb-c}. In the sequel, we would only use the result for the half-integral weight cusp forms, as for those with integral weights, we prove a \textsl{better} quantitative result (see \thmref{partb-m}), which relies on multiplicity-one (Method~2). The details can be found in sections~\ref{ell-icusp},~\ref{ell-hcusp}.

\subsubsection*{Step~2} The remaining step is to treat non-cusp forms. When $n=1$, we are reduced to usual elliptic modular forms; for which we present a new method (Method~2), which is actually robust enough to deal with both cusp and non-cusp forms! See section~\ref{ell-ncusp} for the details. Let us only mention here that this method crucially relies on the multiplicity-one for the newspaces and applies only for integral weights. We therefore assume here that $n$ is odd so that the weight of $h_\mu$ is integral.
But the quality of the quantitative result is better than what can be obtained by Method~1. See \thmref{partb-m}. 

We finally combine the main results from these steps and finish the proof at the end of section~\ref{anaB}, presented in subsection~\ref{concl}. 

To put things into perspective, let us mention that in degree $2$ our proof looks somewhat similar to that in \cite{saha} (here the setting is that of scalar-valued cusp forms) in that we also reduce the question to a suitable Jacobi cusp form, say $\phi$. However there are quite a few interesting differences. 

(i) First, instead of using the Eichler-Zagier map to reduce the question further to half-integral weight elliptic modular forms, we work directly with any of the `primitive' theta components of $\phi$ (i.e., those theta components $h_\mu$ for which $(\mu, 4m)=1$, where $m=$ index of $\phi$). These $h_\mu$ automatically have Fourier expansion supported away from the level, so the analytic treatment becomes easier (cf. \cite[Prop.~5.1]{AD}, \cite[Prop.~3.7]{saha}). More importantly, these primitive theta components are crucial for us since we are led to deal with levels which are squares, and these levels does not satisfy the conditions of \cite[Thm.~2]{AD} or \cite[Thm.~2]{saha}.

(ii) Second, our induction procedure only allows for the index $m$ to be square-free, whereas in \cite{saha} one could take $m$ to be an odd prime. This is not serious when $n=2$, but for higher degrees this is a non-trivial point; it may not be possible to choose a non-zero Fourier-Jacobi coefficient $\varphi_T$ with $T \in \Lambda^+_{n-1}$ and $d_T$ a prime, see \rmkref{p}. However we show in \corref{nton1} that one can always choose such a $\varphi_T \ne 0$ with $d_T$ odd and square-free.

(iii) Third, by choosing $m$ (sticking to $n=2$ for illustration) odd and square-free and invoking \propref{primh}, we avoid the subtlety of the injectiveness (this was a non-trivial difficulty in \cite{AD}) of the Eichler-Zagier map. This injectiveness property ensures smooth passage from Jacobi forms to elliptic modular forms, and is known only when $m$ is a prime and $n=2$. For degrees $n \ge 3$ the theory of Eichler-Zagier maps is not well-developed. We circumvent this by using what we call the `primitive' theta components of $\varphi_T$.
Moreover, we believe that this method should work with suitable modifications for other kinds of modular forms, e.g., Hermitian modular forms.

Concerning an important application of our main result, let us recall the recent work of A. Pollack \cite{Pol}, where a meromorphic continuation of the degree $8$ spinor $L$-function $Z_F(s)$ attached to a Siegel Hecke eigenform $F$ on $\spth$ was proved. Further it was shown that it's functional equation follows under the assumption of nonvanishing of some Fourier coefficient $a_F(T)$ with $T$ corresponding to a maximal order in a quaternion algebra over $\mf Q$ ramified at $\infty$. In the last section of the paper we show, as an application of a variant of our main result, how to remove the aforementioned assumption to get an unconditional result, which may be stated as follows.

Let $\Lambda_F(s)$ denote the completed spinor $L$-function attached to the Hecke eigenform $F$ (see e.g., \cite[p.~2]{Pol} for the description of the Gamma factors).

\begin{thm} \label{poll}
Let $F$ be a non-zero Siegel cuspidal eigenform form on the group $\spt$ of weight $k \ge 3$. Then the spinor $L$-function $Z_F(s)$ attached to $F$ has a meromorphic continuation to $\mf C$, is bounded in vertical strips and satisfies the functional equation $\Lambda_F(s) = \Lambda_F(1-s)$.
\end{thm}
As indicated before, note that the proof of \thmref{poll} does not use \thmref{mainthm} directly, instead uses a variant of it, which is proved in a self-contained manner in the section~\ref{app}. To put things into perspective, let us note here that the analytic properties of the spinor $L$-function for eigenforms on $\spn$ were conjectured by Andrianov (which are of course special cases of Langlands conjectures, note that we are dealing with non-generic automorphic forms here) and they were proved by him when $n=2$ (cf.~\cite{An3}). The meromorphic continuation of the spinor $L$-function for eigenforms on $\spt$ are known by the work of Asgari-Schmidt (cf. \cite{AS}), but obtaining the functional equation of these objects is a delicate matter and was not known for $n \geq 3$. Thus Pollack's work combined with our results from this paper shows the functional equation unconditionally for the first time when $n=3$.

As another application, let us mention that if one studies the standard $L$-function via the Andrianov identity (see \cite{An1}), one has to use a Rankin convolution involving a theta series attached to a quadratic form $T$. It is quite convenient to know from the beginning (from \thmref{mainthm}) that one may choose $T$ to have square-free discriminant (and hence the nebentypus character of the theta series is a primitive quadratic character), see \cite{An1, An2, Cou-Pan} for details.

As a last remark, let us mention that we have not considered the case of higher levels as the content of the paper is already quite technical; but it definitely is an interesting problem to consider. Let us just mention that our methods should also work in this more general setting, but we expect more complicated answers (cf. \cite{AD} for $n=1$). One may have to take into account the Fourier expansions at all cusps simultaneously and one may expect new difficulties concerning primes dividing the level.

In an Appendix (section~\ref{appn}) we lay down the first steps to study the asymptotic properties of non-cusp forms of integral of half-integral weights \textsl{intrinsically} -- i.e., without explicity using a basis of Eisenstein series. We hope that further investigations may shed some light on treating the question on the existence of fundamental Fourier coefficients of non-cuspidal Siegel modular forms of even degree and full level -- something which was not covered in \thmref{mainthm}.

As a side result arising from the Appendix, we prove the following statement: (see \lemref{half-hecke})
\begin{itemize}
\item
{\it an elliptic modular form $f$ of integral or half-integral weight $\kappa \geq 2$ on any principal congruence subgroup whose Fourier coefficients $a_f(n)$ satisfy the bound $a_f(n) \ll_f n^c$ for all $n \ge 1$ and for any fixed $c< \kappa-1$, must be a cusp form.}
\end{itemize}

This result may be of independent interest, and seems to be new for half-integral weights, and for the proof we do not use explicit combinatorics of (linear combinations of) Fourier coefficients of Eisenstein series for such weights. Let us however note that this result should also follow from the arguments in \cite[\S~5]{bo-ko} or by adopting the methods of \cite{BD1} with some technical modifications, 
but it seems these has not been worked out in the half-integral setting yet.

{\small \subsection*{Acknowledgements}
We would like to thank P. Anamby, A. Pollack, R. Schulze-Pillot for useful conversations. S.B. thanks IISc Bangalore and S.D. thanks Universit\"{a}t Mannheim and Alfr\'{e}d R\'{e}nyi Institute for Mathematics, where parts of this work was done, for generous hospitality and for providing excellent working conditions. S.D. also thanks IISc. Bangalore, DST (India) and UGC centre for advanced studies for financial support. During a substantial part of this work, S.D. held a Humboldt Fellowship from the Alexander von Humboldt Foundation and was supported by a MATRICS grant MTR/2017/000496 from DST-SERB, India. The authors are grateful to the referees for their detailed comments and suggestions and for pointing out some errors in an earlier version, which led to improved results.}

\section{Notation and preliminaries} \label{prelim}
\subsection{General notation}
\begin{inparaenum}[(1)]
\item 
Let $\rho \colon \mrm{GL}(n,\mf C)\rightarrow \mrm{GL}(V)$ be a finite 
dimensional (not necessarily irreducible) 
rational representation (a morphism in the sense of algebraic groups) with 
$m=\dim(V)$. 
We call $\rho$  a polynomial representation, if $ \rho$ can be 
realized as a map $\rho \colon \mrm{GL}(n,\mf C)\rightarrow \mrm{GL}(m, \mf C) $ where 
all the coordinate functions $g\longmapsto \rho_{ij}(g)$ are given by 
polynomial functions of the entries of $g$.  
 
For any such $\rho$ there exists a largest integer $k$ such that $\det^{-k}\otimes \rho$ is 
polynomial and we call this $k$ the (determinantal) weight $k(\rho)$. We tacitly use the fact that this weight does not decrease, if we tensor $\rho$
with another polynomial representation or restrict it to some 
$\mrm{GL}(n',\mf C)$ sitting inside $\mrm{GL}(n,\mf C)$ as an algebraic  subgroup. This follows easily by looking at the entries of $\rho(g)$ in any matrix realization of $\rho$.

\item
For a commutative ring $R$ with $1$, we denote by $M_{m,n}( R)$ to be set of $m \times n$ matrices with coefficients in $R$. If $m=n$, we put  $M_{n,n}( R)= M_n(R)$. We denote the transpose of a matrix $M$ by $M^t$. Further, for matrices $A,B$ of appropriate sizes, $A[B]:=B^t AB$. We denote
the $n\times n$ identity matrix over a subring of $\mf C$ by $1_n$. For quantities $a_1,\ldots,a_n \in R$ we denote by $\mrm{diag} (a_1, \ldots,a_n)$ the matrix consisting of the diagonal entries as $a_1, \ldots,a_n$.

Further $\mf Z_q$ denotes the ring of $q$-adic integers for a prime $q$, and $\nu_q$ the $q$-adic valuation.

\item
We define the set of half-integral, symmetric, positive semi-definite matrices by 
\[ \Lambda_n := \{ S= (s_{i,j})  \in M (n, \mf Q ) \mid S=S^t, s_{i,i} \in \mf Z,  s_{i,j} \in \tfrac{1}{2}\mf Z, \text{ and } S \text{ is positive semi-definite}\} \]
and denote the subset of positive definite matrices in $\Lambda_n$ by $\Lambda^+_n$.

\item
For $T$ real and $Z \in M_n(\mf C)$ we define $e(TZ) := \exp(2 \pi i \mrm{tr}(TZ))$, where $\mrm{tr}(M)$ is the trace of the matrix $M$.

\item
Throughout the paper, $\varepsilon$ denotes a small positive number which may vary at different places. Moreover the symbols $A \ll_c B$ and $O_S(T)$ have their standard meaning, implying that the constants involved depend on $c$ or the set $S$. 
\end{inparaenum}
   
\subsection{Siegel modular forms} \label{smfpre}
We denote by 
\[ \mf H_n:=\{ Z\in M_n(\mf C)\, \mid Z=Z^t, \Im(Z)>0\} \] 
the Siegel's upper half space of degree $n$.
The symplectic group $\spnr$ acts on $\mf H_n$ by
$Z\mapsto g \langle Z \rangle=(AZ+B)(CZ+D)^{-1}$; for a polynomial representation $\rho$ 
with values in $\mrm{GL}(V)$ we define the stroke operator action on $V$-valued functions $F$ on $\mf H_n$ by
\[ (F\mid_{\rho}g)(Z):=\rho(CZ+D)^{-1}F(g \langle Z \rangle). \]
A Siegel modular form of degree $n$ and automorphy factor $\rho$
is then a $V$-valued holomorphic function $F$ on $\mf H_n$ satisfying
$F\mid_{\rho}\gamma= F$ for all $\gamma\in \spn$
with the standard additional condition in degree $1$.
We denote by $M^n_{\rho}$ the vector space of all such functions and by
$S^n_{\rho}$ the subspace of cusp forms; if $\rho$ is scalar valued, we write
as usual $M_k^n$ and $S^n_k$ if $\rho=\det^k$. An element $F\in M^n_{\rho}$ has a Fourier
expansion
\[ F(\mc Z)= \sum_{S\in \Lambda_n} a_F(S) e(S \mc Z) \qq\qq ( a_F(S)\in V). \]
If $F$ is cuspidal, then this summation is supported on $\Lambda_n^+$. Sometimes we also write $a(F,S)$ for $a_F(S)$.

The definition above makes sense for arbitrary rational $\rho$,
but (if $\rho$ is irreducible) 
by a theorem of Freitag (\cite{Fr2}) $M^n_{\rho}$ can be nonzero 
only if $\rho$ is polynomial; therefore, we only have to take 
care of polynomial representations.

For $g \in \spne$ we denote by $g^{\downarrow}$ the image of $g$ under the diagonal embedding 
\begin{equation} \label{embed}
  \spne \hookrightarrow \spnr; \q \q g=\begin{psmallmatrix} a & b \\ c & d \end{psmallmatrix} \mapsto  \begin{psmallmatrix} 1 & 0 & 0 & 0 \\ 0 & a & 0 & b \\ 0 & 0 & 1 & 0 \\ 0 & c & 0 & d  \end{psmallmatrix} .  
\end{equation}
We also use the embedding of $\mrm{GL}(n-1,\mf R) \hookrightarrow \mrm{GL}(n,\mf R)$ given by $g \mapsto \begin{psmallmatrix} 1 & 0 \\ 0 &g \end{psmallmatrix}$.

\begin{defi} Let $f \in M_k^n$ and $Z \in \mf H_{n-1}$. The Siegel's $\Phi$-operator is then defined by
\begin{align} \label{phiop}
\Phi (F) (Z) := \underset{ t \rightarrow \infty} \lim F  \left( \begin{smallmatrix} Z & 0 \\ 0 & i t \end{smallmatrix} \right) = \sum_{T \in \Lambda_{n-1}} a_F(\begin{psmallmatrix}  T & 0 \\0 & 0 \end{psmallmatrix} )e(TZ).
\end{align}
\end{defi}
Then it is well known \cite{Fr} that $ \Phi (F)  \in M_k^{n-1}$. Moreover, $F \in M^n_k$ is a cusp form if and only if $F $ is in the kernel of the $\Phi$ operator.

\subsection{Jacobi forms} \label{jacobi}
Throughout this paper, we use a decomposition
for ${\mathcal Z}\in \mf H_n$ into blocks as follows:
\begin{equation}\label{decompo}
\mathcal Z=\begin{pmatrix}\tau &\z\\
\z^t & Z \end{pmatrix}  \qq (\z \in \mf C^{(1,n-1)}, Z\in \mf H_{n-1}).
\end{equation}
Clearly, every $F\in S^n_\rho$ has a 
Fourier-Jacobi expansion with respect to the decomposition above: 
\begin{equation} \label{FJ}
F({\mathcal Z})=\sum_{T\in \Lambda_{n-1}} \varphi_T(\tau,\z) e(TZ).
\end{equation}
The $\varphi_T$ are then ``Jacobi forms'' of automorphy factor $\rho$ and
index $T$, i.e. the functions $\psi({\mathcal Z}):= \varphi_T(\tau,\z)e (TZ)$
on $\mf H_n$ are holomorphic, satisfy $\psi\mid_{\rho} g=\psi$ for all $g\in C_{n, n-1}(\mf Z):= \{\left(\begin{smallmatrix} A & B\\
C & D\end{smallmatrix}\right)\mid (C,D)= 
\left(\begin{smallmatrix} * & 0 &{}& * & D_2\\
0 & 0 & {} & 0 & 1 \end{smallmatrix}\right)\}$,
where $*$ denotes some scalar entries, and satisfy the boundedness condition (Fourier expansion at $\infty$) that
\[ \psi(\mc Z)= \sum_{S = \left( \begin{smallmatrix} n & r/2 \\ r^t/2 & T \end{smallmatrix}\right) \in \Lambda_n} a_\psi(S) e(S \mc Z).\]
Note that this definition of vector-valued Jacobi forms does not agree
with the one in \cite{Zi}; for degree 2 our definition is the same as in \cite{Ibu}. In our set-up, where we work with Fourier-Jacobi coefficients, the above definition is obtained from the corresponding automorphy for Siegel modular forms, and we must use it. In \cite{Zi}, $\rho$ acts only on the Siegel upper-half space variable $\tau$. Unless the automorphy factor is $\rho = \det^k$, the definition given here does not match with that in \cite{Zi} (cf. \cite[p.~785, eq.~(1)]{Ibu}).

The case $\rho=\det^k$ is well known (see e.g.\cite{Kl, Zi}), in particular
scalar-valued Jacobi forms  $\phi_T$ admit a ``theta expansion''
\begin{equation} \label{thetaexp}
\varphi_T(\tau,\z)=\sum_{\mu} h_{\mu}(\tau)\cdot \Theta_T[\mu](\tau,\z)
\end{equation}
with summation over $\mu\in \mf Z^{n-1}/2T\cdot \mf Z^{n-1}$
and
\begin{displaymath} \label{jtheta}
\Theta_T[\mu](\tau,\z)=\sum_{R\in \mf Z^{n-1}}
e^{2\pi i ( T[R+\tilde{\mu}]\tau+ 2 \z T(R+\tilde{\mu})}.
\end{displaymath}
Here we use $\tilde{\mu}:= (2T)^{-1}\cdot \mu$.
We note here that the $h_{\mu}$ are then modular forms of weight 
$k-\frac{n-1}{2}$ 
on some congruence subgroup; the Fourier expansion of $h_{\mu}$ is of shape
\begin{equation} \label{hmu-fe}
h_{\mu}(\tau)=\sum_n a_{\mu}(n-T^{-1}[\mu/2]) e^{2\pi i (n- T^{-1}[\mu/2])\cdot \tau}
\end{equation}
and its Fourier coefficients are given  by 
\begin{equation}
a_{\mu}(n- T^{-1}[\mu/2]) = a_F(\left(\begin{smallmatrix} n & \mu/2\\
\mu/2^t & T\end{smallmatrix} \right) )
\label{relation}
\end{equation}
provided that $\varphi_T$ is the Fourier-Jacobi coefficient of
some scalar-valued Siegel modular form $F\in M_k^n$. We denote the space of (scalar-valued) Jacobi forms of weight $k$ and index $T$ by $J_{k,T}$.

\subsection{Maximal lattices and primitivity} \label{maxi}
When we talk about the `lattice $M$' for some $M \in \Lambda_n^+$, of course we are tacitly identifying $2M$ with the `even' lattice ${\mathcal L}_M:=(\mf Z^n, \mu \mapsto (2M)[\mu])$ inside $\mf Q^n$. We omit the accompanying quadratic form from notation when there is no danger of confusion.
{\bf Henceforth throughout this paper we assume that $d_M$ is odd and square-free}, in particular this implies that $2M$ corresponds to a maximal lattice; in other words there exist no even integral lattice properly containing $\mc L_M$. This can be seen easily: if $\mc L_M \subsetneq \mc L$ for another even lattice $\mc L$ with gram matrix $A$, then $\mc L_M = H\cdot \mc L$ for some $H \in M_n(\mf Z)$, and $d_M = \det(H)^2 d_{ A}$. Clearly $H$ can not be unimodular, and thus $d_M$ could not have been square-free.

Let us recall that the level $\ell_M$ of $2M$ is the smallest $\ell \geq 1$ such that $\ell \cdot (2M)^{-1}$ is even integral. We next compute the level in terms of the (absolute) discriminant.

\begin{lem}
Let $d_M$ be odd and square-free. Then $\ell_M$ is equal to $d_M$ if $n$ is even and equal to $4d_M$ if $n$ is odd. 
\end{lem}

\begin{proof}
When $n$ is even this follows from the facts that $\ell_M \mid d_M, \, d_M \mid \ell_M^{n} $. When $n$ is odd, we have $\ell \mid 4d_M, \, 4d_M \mid \ell^{n} $ which imply $2d_M \mid \ell \mid 4d_M$. To get the exact power of $2$ dividing $\ell$, we appeal to the local theory of quadratic forms. Namely one knows that $(2M)[U_2] = \mbb H\perp \dots \mbb H\perp 2 $ for some $U_2 \in \mrm{GL}_n(\mf Z_2)$ and $\mbb H$ being the hyperbolic plane (see section~\ref{redeg1} for more discussion). Our claim then follows from that facts that the 

(i) levels of $2M$ and $(2M)[U_2] $ are the same;

(ii) for the even quadratic form $\mbb H\perp \dots \mbb H\perp 2 $ over $\mbb Z_2$, the level clearly equals $4$; and

(iii) the level of $2M$ over $\mf Z_2$ divides that over $\mf Z$.
\end{proof} 

\begin{prop} \label{prim}
Let $d_M$ be odd, square-free. There is a $\mu\in \mf Z^n$ such that
$\frac{1}{4} M^{-1}[\mu] $ has exact denominator $d=d_M$ if n is even and $4d$ if $n$ is odd.
\end{prop}
This proposition would be crucial for us, and we would call such $ \mu$ to be \textsl{\textbf{primitive}} in the sequel.
\begin{proof}
To show this, we first note that maximality is a local property (see e.g., \cite{Ki}), and work locally. For any prime $p$, let us denote the (maximal) lattice $\mc L_M \otimes \mf Z_p$ by ${\mathcal L}_{M,p}$ and use the identification ${\mathcal L}_{M,p}:=(\mf Z_p^n, \mu \mapsto (2M)[\mu])$ (or simply by $(\mf Z_p^n, 2M)$) inside $\mf Q_p^n$. We now prove the crucial property in the claim below.

\noindent \textbullet \textbf{\, Claim~1:} For any prime $p$ and $\mu \in \mf Z_p^n$ we have the property:
\begin{equation}
\begin{aligned}
(2M)^{-1}[\mu]\in 2 \mf Z_p & \q \text{  if and only if  } &\mu \in 2M\cdot \mf Z_p^n &  \q (n \geq 1).
\end{aligned}
\label{ganz}
\end{equation}
To prove the claim, let us note that the `if' statement is trivial. In the other direction, for any prime $p$, write $\mu = (2M) \cdot \widetilde{\mu}$ and assume that $\widetilde{\mu} \not \in \mf Z_p^n$. Consider the lattice $\widetilde{\mc L}: = {\mathcal L}_{M,p} + \langle \widetilde{\mu} \rangle$, where we have put $\langle \widetilde{\mu}  \rangle := \mf Z_p \cdot \widetilde{\mu} $. Then as lattices in $\mf Q_p^n$ carrying the even integral quadratic form $2M$ (now viewed over $\mf Z_p^n$), clearly $\mf Z_p^n \subsetneq \widetilde{\mc L}$. We will be done if we can show that $(\widetilde{\mc L}, 2M)$ is even integral; as this will contradict the maximality of $({\mathcal L}_{M,p} , 2M)$.

Let $\nu \in \widetilde{\mc L}$. Writing $\nu:= \beta + c \widetilde{\mu}$ ($\beta \in {\mathcal L}_{M,p}, c \in \mf Z_p$), we see that $(2M)[\nu] \in 2 \mf Z_p$ if and only if $ (2M)[\widetilde{\mu} ]= (2M)^{-1}[\mu] \in 2 \mf Z_p$. Therefore we are done with \textbf{Claim~1}. 

Let us now proceed to prove \propref{prim}.
We choose $\mu\in \mf Z^n$ such that for all $p\mid \det(2M)$ we have $\mu \notin 2M \cdot \mf Z_p^n$. This can certainly be done locally. Note that $2M$ is equivalent over $\mf Z_p$ to the quadratic form $\, \langle * \perp \cdots * \perp *p \rangle \,$ for odd $p$ (e.g., see \cite[Thm.~3.1]{cas}) where the $*$ are units; and similarly for $p=2$, see the proof of \textbf{Claim~2} below. Then by strong approximation we get $\mu \in \mf Z^n$ with the requested properties. 

We claim that such a $\mu$ is primitive. This follows from first showing that $d$ and $4d$ are indeed the largest possible denominators (because of the level of $M$). We therefore can write for some $\alpha = \alpha_p \in \mf Z_p$ (we remind the reader that $d$ is odd)
\begin{equation} \label{denom1}
\frac{1}{4} M^{-1}[\mu]  = \frac{1}{2} (2M)^{-1}[\mu] = \begin{cases}  \frac{\alpha}{d} \q &(p \text{ odd})\\  \frac{\alpha}{4d} \q &(p=2) \end{cases}.
\end{equation}
We have to check that $(\alpha_p,p)=1$.
When $p$ is odd, things are smooth, and the lemma follows from \eqref{ganz} in \textbf{Claim~1} along with our choice of $\mu$ above.

When $p=2$ and $n$ is odd, then we need some care. Namely we observe the following, whose proof is deferred to the end of the proof of this proposition.

\noindent \textbullet \textbf{\, Claim~2:}
\begin{equation} \label{isoeq}
 (2M)^{-1}[\mu]\in  \mf Z_2  \text{  if and only if  } (2M)^{-1}[\mu]\in   2 \mf Z_2 .
\end{equation}  
Granting this claim, we can now finish the proof of \propref{prim} when $p=2$. If in \eqref{denom1} $\alpha=\alpha_2$ is odd we are done. Otherwise putting $\alpha'=\alpha/2 \in \mf Z_2$
we get $ (2M)^{-1}[\mu] = \alpha'/d$. Since $d$ is odd, this implies that $(2M)^{-1}[\mu] \in \mf Z_2$ which means, by \textbf{Claim~2} that $(2M)^{-1}[\mu] \in 2 \mf Z_2$, contradicting \textbf{Claim~1}.

It remains to prove \textbf{Claim~2}.
To do that, we appeal to the dyadic theory of quadratic forms (e.g., see \cite[Lem.~4.1]{cas}) to recall that since $n$ is odd and $\nu_2(\det(2M)) = 1$, $ 2M$ is equivalent to the quadratic form $\, \mbb H \perp \cdots \perp \mbb H \perp 2 \, $ over $\mf Z_2$. Here $\mbb H = \left( \begin{smallmatrix} 0 & 1 \\ 1 & 0 \end{smallmatrix} \right)$ is a hyperbolic plane and $\perp$ means orthogonal direct sum. This shows that (with $\mu = (\mu_1,\ldots, \mu_n)^t$)
\[ (2M)^{-1}[\mu] = \frac{1}{2} \mu_n^2 + 2\mf Z_2, \]
from which our claim follows.
\end{proof}

\section{Part~A: algebraic aspects}

\subsection{Step 1: A non-vanishing property for  Fourier-Jacobi coefficients } \label{partas1}

The statement below seems to be new only when we start
from a vector-valued function. In the scalar-valued case variants have appeared in works of Eichler-Zagier \cite{EZ}, Yamana \cite{Yam}, Ibukiyama \cite{Ibu} and others.

We consider the following situation. Let $\rho \colon \mrm{GL}(n,\mf C) \rightarrow \mrm{GL}(V)$ be a polynomial representation and 
$F \colon \mf H_n\longrightarrow V$ a holomorphic function, not identically zero.
We decompose ${\mathcal Z}\in \mf H_n$ into blocks as in \eqref{decompo}.
We then consider the Taylor expansion of $F$ as a function of $\z=(z_2,\dots ,z_n)\in \mf C^{(1,n-1)}$ as follows. We can write
\begin{equation} \label{taylorex}
F(\mc Z)= \sum_{\lambda} F_{\lambda}(\tau,Z) \z^{\lambda},
\end{equation} 
where $\lambda=(\lambda_2,\dots\lambda_n)\in {\mf N}^{n-1}$ is a polyindex and
$\z^{\lambda}:=z_2^{\lambda_2}\dots z_n^{\lambda_n}$.

We put $\nu=\nu(\lambda)=\sum_{i=2}^n \lambda_i$ and
\[\nu_0 :=\min\{\nu(\lambda)\,\mid\, F_{\lambda}\not=0\}. \]
Then we look at all the Taylor coefficients of homogeneous degree $\nu_0$ and we study a polynomial in variables $X_2,\dots, X_n$ of homogeneous degree $\nu_0$:
\begin{equation} \label{Fo}
F^o(\tau,Z):=\sum_{\lambda: \nu(\lambda)=\nu_0} F_{\lambda}(\tau,Z) X_2^{\lambda_2},\dots X_n^{\lambda_n}. 
\end{equation}
We may view $F^o$ as a function on $\mf H\times \mf H_{n-1}$ with values
in $V\otimes \mf C[X_2,\dots , X_n]_{\nu_0}$, where we denote by $\mf C[X_2,\dots , X_n]_{\nu_0}$ the $\mf C$-vector space of homogeneous polynomials of degree $\nu_0$. 

For an integer $m \ge 1$, we denote by $\mrm{Sym}^{m}$ the symmetric $m$-th power representation of $\glnc$ realised in the vector space of homogeneous polynomials over $\mf C$ of degree $m$:
\begin{align*}
\mrm{Sym}^{m} \colon \glnc \rightarrow \mrm{GL}(\mf C[Z_1,\dots &, Z_n]_m), \q g \mapsto  g \cdot f \q (f \in \mf C[Z_1,\dots , Z_n]_m)\\
(g \cdot f)(Z_1,\dots , Z_n) &:= f( (Z_1,\dots , Z_n) \cdot g).
\end{align*}
Recall the embedding $\mrm{Sp}(n-1,\mf R) \hookrightarrow \mrm{Sp}(n,\mf R)$ given by $g \mapsto g^{\downarrow}$, from \eqref{embed}.

\begin{prop} \label{vec}
Let the setting be as above. Then for any $g=\left( \begin{smallmatrix} a & b \\ c & d \end{smallmatrix}\right) 
\in \mrm{Sp}(n-1,\mf R)$
\begin{equation}
\left(F\mid_{\rho}g^{\downarrow}\right)^o (\tau,Z)= 
\rho \left(  \left(\begin{smallmatrix} 1 & 0\\
0 & cZ+d\end{smallmatrix}\right) \right)^{-1}\otimes \mrm{Sym}^{\nu_o}(cZ+d)^{-1} 
F^o(\tau, g \langle Z \rangle).\label{taylor} 
\end{equation}    
In particular if $F \in M^n_\rho$, then $F^o$, viewed as a function of $Z$, is in $M^{n-1}_{\rho'\otimes Sym^{\nu_0}}$,
where $\rho'$ is the restriction of $ \rho$ to $GL(n-1)\hookrightarrow GL(n)$ (cf. subsection~\ref{smfpre}). Moreover if $F$ is cuspidal, then $F^o$ is also cuspidal. Further if $F \neq 0$, then for some $\tau=\tau_0$, $F^o(\tau_0,Z)$ is non-zero as a function of $Z$.
\end{prop}

\begin{proof}
We recall that (see \cite{Fr})
\[g^{\downarrow}<{\mathcal Z}>= \begin{pmatrix}
\tau- \z(cZ+d)^{-1}c\z^t & \z(cZ+D)^{-1}\\
(cZ+D)^{-t}\z^t & g \langle Z \rangle \end{pmatrix} \]
Then we compute from \eqref{taylorex} that
\begin{equation} \label{Fg}
F\mid_{\rho}g^{\downarrow}=
 \rho( \begin{psmallmatrix}
   1 & 0\\
c\z^t & cZ+d \end{psmallmatrix})^{-1}\sum_\lambda
F_{\lambda}(\tau- \z(cZ+d)^{-1}c\z^t, g \langle Z \rangle) \left(\z(cZ+d)^{-1}\right)^{\lambda}. 
\end{equation}
We pick out the contributions to  $\z^{\lambda}$ with $\nu(\lambda)=\nu_0$:
Due to the minimality of $\nu_0$, all summands on the right hand side have degree $\geq \nu_0$
as polynomials in $\z$.

Now putting $\mf h=\mf h(\z) := - \z(cZ+d)^{-1}c\z^t$, and Taylor expanding around $\tau$ (with $\z$ in a sufficiently small neighbourhood of $0$), we get
\[  F_\lambda(\tau+\mf h,Z) = F_\lambda(\tau,Z) + O(\mf h) ; \]
here $O(\mf h)$ means a multiple of $\mf h$. Thus by minimality (see \eqref{Fo}) only $F_{\lambda}(\tau, g \langle Z \rangle)$ may contribute.

Moreover only $\rho(\left( \begin{smallmatrix}  1 & 0\\
0 & cZ+d \end{smallmatrix}   \right))^{-1}$ has to be considered. To see this note that
\begin{equation} \label{rhoz}
\rho( \begin{psmallmatrix}
   1 & 0\\
c\z^t & cZ+d \end{psmallmatrix})^{-1} = \rho( \begin{psmallmatrix}
   1 & 0\\
-(cZ+d)^{-1}c\z^t & 1_{n-1} \end{psmallmatrix}) \cdot \rho( \begin{psmallmatrix}
   1 & 0\\
0 & cZ+d \end{psmallmatrix})^{-1}.
\end{equation}
Let us observe here that considering the homogeneous decomposition of the quantity $\rho( \begin{psmallmatrix}
   1 & 0\\
-(cZ+d)^{-1}c\z^t & 1_{n-1} \end{psmallmatrix}) $ as a polynomial in $\mk z$ we can write
\[
\rho( \begin{psmallmatrix}
   1 & 0\\
-(cZ+d)^{-1}c\z^t & 1_{n-1} \end{psmallmatrix}) = 1_{n} + P(\mk z), \]
where $P(\mk z) \in M(n, \mf C) \otimes \mf C[\mk z]$ has entries which are polynomial in $\mk z$ without a constant term.
Since multiplication of the polynomial expression $\rho(\left( \begin{smallmatrix}  1 & 0\\
0 & cZ+d \end{smallmatrix}   \right)^{-1} F_\lambda(\tau,Z)\left(\z(cZ+d)^{-1}\right)^{\lambda}$ in $\z$ by $P(\mk z)$ can only increase it's degree in $\z$; which however already has degree $\nu_0$, it follows that only $\rho(\left( \begin{smallmatrix}  1 & 0\\
0 & cZ+d \end{smallmatrix}   \right))^{-1}$ has to be considered. This proves the automorphy of $F^o$.

The assertion about cuspidality of $F^o$ follows easily by looking at its Fourier expansion as a function of $Z$ and that about $F^o(\tau_0,Z)$ being not identically zero for some $\tau_0$ is trivial. This proves the proposition.
\end{proof}

\begin{rmk} \label{p}
The situation in \cite{saha} was very special: a result of Yamana \cite{Yam} states that a non-vanishing Siegel cusp form  of level one always has a non-vanishing Fourier coefficient supported on a primitive (binary) quadratic form. From this one could immediately get a non-vanishing Fourier-Jacobi coefficient of prime index. The argument here relies on the fact that a primitive binary quadratic form always represent infinitely many primes.

In order to pursue this procedure in our situation, say for degree $3$, we would need to prove that every primitive ternary quadratic form represents a binary quadratic form whose determinant is square-free. Unfortunately this is not  true in general. We give a counterexample.

From the local theory of ternary quadratic forms, we can find a
ternary quadratic form $T$, which for an odd prime $p$
is equivalent over $\mf Z_p$ to a form 
\[ \mrm{diag}( \epsilon,p^2,\mu\cdot p^2)\]
with $\epsilon,\mu\in \mf Z_p^{\times}$.
Clearly, all binary quadratic forms integrally represented by such $T$ have determinant divisible by $p^2$. \qed
\end{rmk}

Let us now look at the Fourier-Jacobi expansion of $F$ from \eqref{FJ}. Let $T \in \Lambda^+_{n-1}$ and $\varphi_T(\tau, \mk z)$ be the Fourier-Jacobi coefficients of $F$. From the definition of $F^o(\tau_0,Z)$ (cf. \eqref{Fo}) we see that
\begin{equation}\label{n-1fj}
a_{F^o}(T) = c \cdot \sum_{\nu(\lambda)=\nu_0} \frac{\partial^\lambda}{\partial \, \mk z^\lambda} \varphi_T(\tau_0,\mk z) |_{\mk z =0} 
\end{equation}
where $c$ is a non-zero constant independent of $F$ or $T$.

Clearly if $a_{F^o}(T) \ne 0$, then $\varphi_T \ne 0$. Thus the expression \eqref{n-1fj} provides us our avenue for carrying out an induction argument.

\begin{cor} \label{nton1}
Let $C \in \mf N$ be given. Assume that \thmref{mainthm} holds for
all non-zero forms in $M^{n-1}_\theta$ for all polynomial representations $\theta$ of $\mrm{GL}(n-1,\mf C)$ with $k(\theta)\ge C$. Then for any polynomial representation $\rho$ with $k(\rho)\ge C$,
all non-zero $F \in M^n_\rho$ has (infinitely many)
non-vanishing Fourier-Jacobi coefficients $\varphi_T$ with $T$ of size $n-1$ and square-free odd discriminant.
\end{cor}

\subsection{Step 2: Reduction to the case of scalar-valued Jacobi forms} \label{jsca}
The Fourier-Jacobi coefficients of (the vector-valued) $F$ do have a theta expansion, which is more complicated than
in the scalar-valued case. We do not pursue writing this down, but our aim here is to show that we may choose a nonzero (vector-)
component of $\varphi_T$, which  
behaves as a Jacobi form in the scalar-valued case. Throughout this section, we use the block decomposition \eqref{decompo}.


We consider the  two transformation laws responsible for the theta expansion. The first one is
\[ \varphi_T(\tau,\z+r)=\varphi_T(t,\z)\]
for any $r\in \mf Z^{(1,n-1)}$; 
this comes from the transformation law of $F$ for the matrix 
$\begin{psmallmatrix} 1_n & S\\
0 & 1_n\end{psmallmatrix}$ with $S= \begin{psmallmatrix} 0 & r\\
r^t & 0_{n-1}\end{psmallmatrix}$.
The second transformation law is obtained from
$M=\begin{psmallmatrix} U^t & 0\\
0 & U^{-1}\end{psmallmatrix} \in \spn $ with $U= \begin{psmallmatrix} 1 & \ell\\
0 & 1_{n-1}\end{psmallmatrix}$; here $\ell \in \mf Z^{(1,n-1)}$ and
it gives
\begin{equation}\varphi_T(\tau,\tau\ell+\z)=\rho(U^{-1})\varphi_T(\tau,\z)
e(- \left(T[\ell^t]\tau +2 \cdot T   \ell^t \z \right))\label{Lie}.\end{equation}

Let $\Delta_n \subset \glnc$ be the subgroup of all upper triangular matrices. Since $\Delta_n$ is a connected, solvable algebraic group, by the Lie-Kolchin theorem (cf. \cite[Thm.~10.5]{borel}), there exist a basis of $V$ such that the set $\rho(\gamma)$ is upper triangular for all $\gamma \in \Delta_n$. Thus without loss, from now on we assume that $V= \mf C^m$ and that all elements of $\rho(\Delta_n)$ are upper triangular.

We view $\varphi_T$ as a $\mf C^m$-valued function,
$\varphi_T= (\varphi_T^{(1)},\dots ,\varphi_T^{(m)})^t$.
We put
\[r:= \max \{i\,\mid 1\leq i\leq m, \varphi_T^{(i)}\not=0\}.\]
For this $r$, the equation (\ref{Lie}) reads
\[\varphi^{(r)}_T(\tau,\tau\ell+\z)=\varphi^{(r)}_T(\tau,\z)
e(- \left(T[\ell^t]\tau +2 \cdot T   \ell^t \z \right)).
\]
We must finally check the transformation law for $\sltwo$. Namely:
\[
e(\frac{-c}{c\tau+d}\z T\z^t) \, \rho \left(\begin{pmatrix} c\tau+d & c\z\\
0 & 1_{n-1}\end{pmatrix} \right)^{-1} \varphi_T \left(\frac{a\tau+b}{c\tau+d}, \frac{\z}{c\tau+d} \right)= \varphi_T(\tau,\z)\]
gives (when applied to $\varphi_T^{(r)}$) the requested transformation property with weight $k'$ given by
\[\rho(\mrm{diag}(\lambda,1\dots , 1))=  (g_{ij}(\lambda)) \]
with $\lambda\longmapsto g_{rr}(\lambda)$ being a (polynomial) character of $\mrm{GL}(1, \mf C)$, i.e. $g_{rr}(\lambda)=\lambda^{k'}$ for some integer $k'\geq 0$. This follows by looking at the $r$-th components on both sides of the above equation and here we crucially use the property of $\rho(\Delta_n)$ to be upper triangular.

This means that $\varphi_T^{(r)}$  is a non-vanishing scalar-valued Jacobi form of weight $k' \geq k(\rho)$. Summarizing, we have shown that:

\begin{prop} \label{sca-j}
If $\varphi_T \neq 0$ is a vector-valued Jacobi form of index $T$ with respect to $\rho$ in the sense of subsection~\ref{jacobi}, then there exist a component of $\varphi^{(r)}_T$ of $\varphi_T$ which is a scalar-valued Jacobi form (of integral weight $k' \geq k(\rho)$) in the sense of \cite{EZ}.
\end{prop}

\subsection{Step 3: On primitive components of theta expansions}

We now work with the scalar-valued Jacobi form $\varphi_T^{(r)}$ from the previous section. More generally we prove the existence of `primitive' theta components of any such form whose index has absolute discriminant odd and square-free. Note the switch from $n-1$ to $n$ in this subsection, for convenience. We would later apply the result of this subsection to a $T \in \Lambda_{n-1}^+$. The following result is independent of the previous considerations.

\begin{prop} \label{primh}
Let $M \in \Lambda_n^+$ be such that $d=d_M$ is odd and square-free.
Let 
\[\phi_M(\tau,z)=\sum_{\mu} h_{\mu}(\tau) \Theta_M[\mu](\tau,z) \in J_{k, M} \]
be the theta decomposition of $\phi_M$ (see e.g., \cite[p.~210]{Zi}). Assume that $h_{\mu}=0$ for all
primitive $\mu$. Then $\phi_M=0$.
\end{prop}

\begin{rmk}
\begin{enumerate}

\item[(1)]
This property is weaker than irreducibility of the
theta repesentation; note that in the case of scalar index $m$
Skoruppa \cite{Sko} showed that irreducibility 
holds only for index $m=1$ or a prime $p$.

\item[(2)]
This result is expected to hold only when $M$ is a maximal lattice, for example when $n=1$ and $d_M$ is not square-free, there are non-zero Jacobi forms with vanishing primitive $h_\mu$ which `come' from index-old forms, see \cite[Lem.~3.1]{SZ}.
\end{enumerate}

\end{rmk}

\noindent{\it Proof of \propref{primh}.}
The proof is rather long, and has been divided into several parts for convenience.
Let $\phi_M$ be a Jacobi form satisfying the assumption of the proposition.
Then for all primitive $\mu$, by considering $h_{\mu} |\begin{psmallmatrix}
0 & -1\\
1 & 0\end{psmallmatrix}$, we obtain the relation
\[\sum_{\nu\in \mf Z^n/2M \cdot\mf Z^n } e(  \frac{1}{2} \langle\nu,\mu \rangle)  h_{\nu}=0,\]
where for vectors $\nu,\mu$ as above, we have put $ \langle\nu,\mu \rangle = \nu^tM^{-1}\mu$.

By applying translations $\tau\longmapsto \tau+t$ with $t\in \mf Z$ the equation above becomes
\[\sum_{\nu\in \mf Z^n/2M\mf Z^n } e ( \frac{1}{2} \langle\nu,\mu \rangle +   \frac{1}{4} \langle\nu,\nu \rangle \,  t)
\cdot h_{\nu}=0. \]

We observe that $\nu$ and $\nu'$ define the same character 
$t\longmapsto e( \frac{1}{4} \langle\nu,\nu \rangle \, t) $
of $(\mf Z,+)$
if and only if
\[
\frac{1}{4} \langle\nu,\nu \rangle  - \frac{1}{4} \langle\nu',\nu' \rangle  = \frac{1}{4}M^{-1}[\nu]-\frac{1}{4}M^{-1}[\nu']\in \mf Z.
\]
Using the linear independence of pairwise different characters
we get a refined system of equations: we fix some $\nu^o\in \mf Z^n$ and are led to consider 
\begin{equation} \label{hmueq}
\sum_{\substack{  \nu\in \mf Z^n/2M\mf Z^n \\
\frac{1}{4}M^{-1}[\nu]-\frac{1}{4}M^{-1}[\nu^o]\in \mf Z } } e( \frac{1}{2} \langle\nu,\mu \rangle  ) \cdot h_{\nu}=0.
\end{equation}
Only the case of imprimitive $\nu_0$ is of interest here, since for primitive $\nu_0$, all $\nu$ appearing in the sum \eqref{hmueq} will also be primitive, and for these the $h_{\nu}$ are zero anyway. 

\noindent \te \textit{{\bf Claim~1:}} For all fixed $\nu^o$ the matrix $\left( e( \frac{1}{2}  \langle \nu,\mu \rangle )\right)_{\nu,\mu} $ (with $\mu$ varying over primitive vectors and $\nu$ varying over all vectors as in \eqref{hmueq}) is of maximal rank (equal to the cardinality of the set of $\nu$ occurring here).

Let us now grant ourselves the claim above and show how to finish the proof of \propref{primh}. Indeed, as noted above, we only need to show that all imprimitive $h_\mu$ are zero. The condition $\nu \sim \nu^o$ if and only if $\frac{1}{4}M^{-1}[\nu]-\frac{1}{4}M^{-1}[\nu^o]\in \mf Z$ defines an equivalence relation on the set of imprimitive indices, and \eqref{hmueq} along with {\bf Claim~1} just says that all the $h_\nu=0$ for all $\nu$ in the equivalence class of $\nu^o$. But since $\nu^o$ can be any arbitrary imprimitive index, we are done.

\subsubsection{Reduction to degree one} \label{redeg1}

Our aim is to show that we can reduce everything to the case of degree $1$, in other words, we show next that it is enough to prove the {\bf Claim~1} above when $M,\mu,\nu$ are scalars. The idea is to choose representatives of $\mf Z^n/2M  \cdot \mf Z^n$ which are similar to those for $n=1$. We give all details when $n$ is odd and indicate the main points for the other case.

\noindent \te \textbf{The case $n$ odd.} First of all, we may find $U\in \mrm{SL}(n,\mf Z)$
such that $\widetilde{M}:= (2M)[U]$ satisfies for $f \geq 2$
\begin{equation} \label{podd}
\wm \equiv \mrm{diag}(*,\dots *, \zeta d) \bmod d^f
\end{equation} 
where $*,\zeta$ are units $\bmod d$
and 
\begin{equation}\label{pev}
\wm\equiv \mbb H\perp \dots \mbb H\perp 2\bmod 2^f,
\end{equation}
where 
$\mathbb H$ denotes the hyperbolic plane $\begin{psmallmatrix} 0 & 1\\
1 & 0\end{psmallmatrix}$. 
Indeed, from the local theory of quadratic forms (see e.g., \cite[Chap.~8, Prop.~3.1, Lem.~4.1]{cas}) we can find for every $q \mid 2d$, matrices $U_q \in \mrm{SL}(n, \mf Z_q)$ such that (since $n$ is odd and $2d$ is square-free), 
\begin{equation}
(2M)[U_q] = \begin{cases}    \mrm{diag}(*,\dots *, q)  &\text  { if } q \neq 2  \\   \mbb H\perp \dots \mbb H\perp 2      &\text  { if } q = 2 \end{cases}.
\end{equation}
The reader may note that a priori (with the convention in \cite{cas}) we can only get a $V_q \in \mrm{GL}(n, \mf Z_q)$ with the above property; but we can assume $V_q \in \mrm{SL}(n, \mf Z_q)$ by multiplying $V_q$ on the left with the matrix $\mrm{diag}(\det(V_q)^{-1},1,\cdots1)$ without loss.

Then by strong approximation for $\mrm{SL}(n)$, we may find $U \in \slnz$ such that $U \equiv U_q \bmod q^f$ for any $f \ge 1$ for any prime $q \mid 2d$. This $U$ works. To get the statement \eqref{podd}, we use the Chinese remainder theorem for the moduli $q^f$ with $f \geq 2$.

As representatives of $\mf Z^n/ \wm \mf Z^n$
we may choose
\begin{equation} \label{repns}
\wnu:=\{ (0,0, \ldots,\wnu_n)^t \,\mid \, \wnu_n \bmod 2d\}.
\end{equation}
Using \eqref{podd}, \eqref{pev} we see that indeed these are pairwise inequivalent by checking locally and noting that the cardinality of this set is the right one. 

Let now $\mbb M=(m_{i,j})$ be the adjoint of $\wm$, so that $\wm \mbb M = 2d \cdot I_n$. Then for $\wnu,\wmu \in \mf Z^n/ \wm \mf Z^n$ (assuming they are in the nice form as in \eqref{repns}), we see that
\begin{equation} \label{madj}
\wnu^t \wm^{-1} \wmu = \frac{1}{2d} \wnu_n  m_{n,n} \wmu_n
\end{equation}
We claim that $(m_{n,n}, 2d)=1$. To see this, we multiply \eqref{podd}, \eqref{pev} by $\mbb M$ on both sides and compare the resulting congruences to obtain that
\[ \zeta d m_{n,n} \equiv 2d \bmod d^f, \q  2 m_{n,n}  \equiv 2d \bmod 2^f,\]
which clearly implies our above claim since $f \ge 2$.

Furthermore for $\wnu,\wnu^o \in \mf Z^n/ \wm \mf Z^n$, we see that
\begin{equation} \label{newrel1}
\frac{1}{2} \wm^{-1}[\wnu]- \frac{1}{2} \wm^{-1}[\wnu^o]\in \mf Z
\, \, \text{  if and only if  } \, \, \wnu_n^2\equiv (\wnu_n^o)^2\bmod 4d,
\end{equation}

Also it is clear from \eqref{madj} along with $(m_{n,n}, 2d)=1$ that $\wnu$ is primitive (with respect to $\wm$) if and only if $(\wnu_n,2d)=1$. 

These considerations allow us to reduce the case of odd $n$ to $n=1$ with $d'=2d$ as follows. We note that
\begin{equation} \label{newrel3}
  \nu^t (2M)^{-1} \mu = \nu^t U \wm^{-1} U^t \mu = \wnu^t \wm^{-1} \wmu, 
\end{equation} 
whence we make a change of variables $\wnu= U^t \nu, \wmu = U^t \mu$; and observe that as $\nu,\mu$ vary over $\mf Z^n/2M  \cdot \mf Z^n$, $\wnu,\wmu$ vary over $\mf Z^n/ U^t (2M) \cdot \mf Z^n = \mf Z^n/ \wm \mf Z^n$. Clearly $\nu$ is primitive for $M$ if and only if $\wnu$ is so for $\wm$. Moreover the condition on $\nu$ in \eqref{hmueq} can be seen to be exactly the one in \eqref{newrel1} upon using \eqref{newrel3}.  

The reduction to $n=1$ is now clear from \eqref{madj} and \eqref{newrel1} where in \eqref{madj} we make a change of variable $\wmu \mapsto m_{n,n}^{-1} \wmu \bmod {2d}$. We would spell this out explicitly at the end of this subsection after taking care of the analogous case of $n$ being even, see {\bf Claim~2} below.

\noindent \te \textbf{The case $n$ even.}  
We may find, arguing as in the previous case, an $U\in \mrm{SL}(n,\mf Z)$
such that $\wm:= (2M)[U]$ satisfies 
\[\wm \equiv \mrm{diag}(*,\dots *,d) \bmod d^f\]
where $*$ are units $\bmod d$
\[\wm\equiv \mbb H\perp \dots \perp \mbb H\bmod 4^f \q \text{ or } \q  \wm\equiv \mbb H\perp \dots \mbb H \perp {\mathbb F} \bmod 4^f,\] 
where $\mathbb F= \begin{psmallmatrix} 2 & 1 \\
1 & 2\end{psmallmatrix}$.

As representatives of $\mf Z^n/ 2M\mf Z^n$ we may choose
\[{\bf \nu}:=\{ ( 0, 0, \ldots, \nu_n)^t
\,\mid \, \nu_n \bmod d\}\] 
and $\nu$ is primitive if and only if $(\nu_n,d)=1$.
Furthermore,
\[
e( \frac{1}{2} \nu^t M^{-1}\mu )= e( \frac{1}{d} \nu_n m_{n,n} \mu_n )
\]
where the matrix adjoint of $\wm$ is $(m_{i,j})$ and
\[\frac{1}{4}M^{-1}[\nu]- \frac{1}{4}M^{-1}[\nu^o]\in \mf Z
\iff \nu_n^2\equiv (\nu_n^o)^2\bmod d\]

These considerations allow us to reduce the case of even $n$ to $n=1$ with $d'=d$. 

Summarizing, we now have to prove the following.

\noindent \te \textit{{\bf Claim~2:}}
Let $d$ be a square-free odd positive integer,
$d=p_1\cdot \dots p_t$.
To cover also even numbers we put (with the convention that $d'=2$ if $t=0$)
\[d':= \begin{cases}  d & \text{  if  } n \text{  is even}\\
2d & \text{  if  }  n \text{   is odd}  \end{cases} .\] 
We assume that $d'>1$.
We fix $\nu_0\bmod d'$. Then the following matrix has maximal rank: 
{\small
\begin{equation}
\left( e( \frac{\mu \nu  }{d'} ) \right)_{ \begin{array}{c}\mu \bmod d',(\mu,d')=1\\
\nu \bmod d', \nu^2\equiv \nu_0^2\bmod d'\end{array}   }.
\end{equation} }

\subsubsection{Proof of {\bf Claim~2}}
By the chinese remainder theorem, the set
\[\{\nu \bmod d'\, \mid\,\nu^2\equiv \nu_0^2\bmod d'\}\]
has cardinality $2^{t'}$, where $t'$ is the number of odd 
primes dividing $d$ and not dividing $\nu_0$;
note that in the case of even $d'$ we might as well describe the congruence by 
$\nu^2\equiv \nu_0^2 \bmod 4d$.

\begin{lem}
With all the conditions from above, we claim  that such a  matrix always has maximal rank (equal to the number of columns $2^{t'}$).
\end{lem}

To prove the above lemma, we argue by induction on $t$. For $t=0$ and $d'=2$ the matrix in question is just a non-zero scalar. For $t=1$ we consider two cases:

\noindent \te{\bf   Case I: $t=1, d'=p$.}
If $\nu_0\equiv 0\bmod p$ the matrix in question is a nonzero column.

Now we look at $p\nmid \nu_0$. We have to consider $\nu=\pm \nu_0$; with a $\mu$ still to be determined so that the matrix of size $2$ as below
\[
\begin{pmatrix} 
e( \frac{ \nu_0 }{p} ) & e(-\frac{ \nu_0 }{p} )   \\
e( \frac{ \mu \nu_0 }{p} )  & e( -\frac{ \mu \nu_0 }{p} )  
\end{pmatrix}
\] 
whose determinant is equal to
\[ 
e( \frac{ (1-\mu) \nu_0 }{p} ) \left(  1- e( -\frac{ 2(1-\mu) \nu_0 }{p} )   \right)
; \] 
should have the determinant non-zero. This clearly implies the maximality of the rank of the original matrix.

We may choose any  $\mu$, coprime to $p$ and different from $1$. This settles case~I.

\noindent \te{\bf  Case II: $t=1, d'=2p$.}
This works similarly:
if $\nu_0\equiv 0\bmod p$, the matrix is again just a nonzero column.

Now we assume $p\nmid \nu_0$: we have to consider $\nu=\pm \nu_0$; and
with a $\mu$ still to be determined we look at the matrix of size $2$:
\[
\begin{pmatrix} 
e( \frac{ \nu_0 }{2p} ) & e(-\frac{ \nu_0 }{2p} )   \\
e( \frac{ \mu \nu_0 }{2p} )  & e( -\frac{ \mu \nu_0 }{2p} )  
\end{pmatrix}.
\] 
The determinant is equal to
\[ 
e( \frac{ (1-\mu) \nu_0 }{2p} ) \left(  1- e( -\frac{ (1-\mu) \nu_0 }{p} )   \right).
\] 
We may choose any $\mu$ coprime to $2p$ and different from $1$. Thus the lemma follows in this case as well. 

\noindent \te{\bf Induction step $t\longmapsto t+1$ with $t\geq 1$:}

We write $q$ for the prime $p_{t+1}$. We decompose $\nu_0 \bmod d'q$ as
\[\nu_0= d'\nu_0'+q\nu_0^{''}\]
with $\nu_0'\bmod q$ and $\nu_0^{''}\bmod d'$
and similarly for $\nu$ and $\mu$.
Then
\[\frac{\mu\nu}{d' q}= \frac{ (d'\mu'+q\mu^{''})(d'\nu'+q\nu^{''})}{d'q}
=
\frac{d'\mu'\nu'}{q}+\frac{q\mu''\nu''}{d'}\bmod \mf Z.
\]
The matrix attached to $d'q$ and $\nu_0$ is then the tensor product (``Kronecker product'')
of the matrices attached to $q$ and $d'\cdot \nu_0'$ and attached to $d'$ and $q\cdot \nu_0''$. The induction step follows from well known property of Kronecker products that the rank of $A\otimes B$ is the product of the ranks of $A$ and $B$.

This finishes the proof of {\bf Claim~2} and hence also that of \propref{primh}. \QEDB

The following lemma, which will be used later, implies that when $k$ is even, \textsl{all} the theta components of a non-cuspidal scalar-valued Jacobi form of index $M$ with $d_M$ odd, square-free, are also non-cuspidal. 
\begin{lem} \label{jncusp}
Let $M \in \Lambda_n^+$ be such that $d=d_M$ is odd and square-free. Suppose that $\phi \in J_{k,M}$ be non-cuspidal and that $k > n+2$ is even. Then all the theta components $h\mu$ of $\phi$ are also non-cuspidal.
\end{lem}
In particular for a non-zero $\phi$ as above, there always exist a non-zero, non-cuspidal primitive theta component. Also we think that probably \lemref{jncusp} holds for all discriminants even in the vector-valued case (proving either of which however seems non-trivial).

\begin{proof}
Let $J_{k,M}^{\mc E}$ and $J_{k,M}^{cusp}$ be the spaces of Eisenstein series and cusp forms respectively.
We write $\phi = \phi_{\mc E} + \phi_{c}$ where $\phi_{\mc E}  \in J_{k,M}^{\mc E}$ and $\phi_{c} \in J_{k,M}^{cusp}$. Since $\phi$ is not a cusp form, $\phi_{\mc E} \neq 0$. 

Let us now recall from \cite[Lem.~3.3.14]{AJ} that
\[ \dim J_{k,M}^{\mc E}= \frac{1}{2} \left( \# \mrm{Iso}(\mc L_M^{\#}/\mc L_M) +(-1)^k \# \{\ \gamma \in \mrm{Iso}(\mc L_M^{\#}/ \mc L_M) \mid 2 \gamma \in \mc L_M   \}  \right)  .\]
where $\mc L_M$ is the lattice associated with $M$ (cf. section~\ref{maxi}) and  $\mrm{Iso}(\mc L_M^{\#}/\mc L_M)$ is the set of isotropic elements in the discriminant form $\mc L_M^{\#}/\mc L_M$ associated to $M$. It is now easy to check (note the normalisation of the quadratic form on \cite[p.~12]{AJ}) for $M$ as in the lemma that $\mrm{Iso}(\mc L_M^{\#}/\mc L_M)$ is just trivial, and this follows precisely from \eqref{isoeq} if we note that $\mc L_M^{\#} = ((2M)^{-1} \mf Z^n, \mu \mapsto (2M)[\mu])$. 

Thus $J_{k,M}^{\mc E} = \mf C\{  E_{k,M,0} \}$ in the notation of \cite{AJ} as $k$ is even. For the absolute convergence of $E_{k,M,0}$ we need here that $k > n/2 +2$.
Now from the results of \cite[Theorem~10]{Bo} $E_{k,M,0}$ appears as the $M$-th Fourier-Jacobi coefficient of the Siegel Eisenstein series, say $E^k_{n+1}$, of degree $n+1$. Here we need to assume that $k>n+2$. Indeed, Theorem~10 in \cite{Bo} describes explicitly the Fourier-Jacobi expansion of any
Siegel Eisenstein series.
If $M$ is maximal, the summation over $w_1$ there is trivial,
which implies our claim.
It is well-known that \textsl{all} the Fourier coefficients of $E^k_{n+1}$ are non-zero (see eg. \cite[Theorem, p.~115]{Ki1}). Thus all the Fourier coefficients of $E_{k,M,0}$ are non-zero -- and thus so are all its theta components.
\end{proof}

\subsection{Formulation of the desired properties of \texorpdfstring{$h_{\mu}$}{aa}   }
\label{2adjust}
We consider a nonzero  $F \in M^n_{\rho}$ and choose (by induction hypothesis)
a $T\in \Lambda_{n-1}^+$ with $d_T$ odd and square-free such that $\varphi_T$ is non-zero (cf. \corref{nton1}). Then we choose by \propref{sca-j} a suitable non-zero component $\varphi_T^{(r)}$ of $\varphi_T$, which is a scalar valued Jacobi form.

Let $(h_0, \dots h_{\mu},\dots )$ denote the components of the theta expansion of $\varphi_T^{(r)}$ (see \eqref{thetaexp}). By \propref{primh}, we get hold of a primitive $\mu$ such that $h_\mu \neq 0$. {\bf We work with this $h_\mu$ (of weight $k'-\frac{n-1}{2}$) for the rest of the paper}. 

The basic starting point is an equation of type
\begin{equation}
{\det}_n \mc T= {\det}_n( \begin{pmatrix} \ell & \frac{\mu}{2}\\
\frac{\mu^t}{2} & T\end{pmatrix})=
(\ell- \frac{1}{4}T^{-1}[\mu^t])\cdot {\det}_{n-1}(T) \label{start}
\end{equation}
for the determinant $\det_n(\mc T)$ of a half-integral matrix ${\mathcal T}$ of size $n$ occurring on the left hand side. 

Recall that $d_{\mc T}$ is the (absolute) discriminant of $2\mc T$ and similarly for $T$.

\subsubsection{Case $n$ is odd} We should multiply the equation \eqref{start} above by $2^{n-1}$:
\[d_{\mathcal T}= (\ell- \frac{1}{4} T^{-1}[\mu])\cdot d_T.\] 
The first factor has exact denominator $d_T$ since $\mu$ is primitive (see \propref{prim}). In the next section ({\it Part~B}) we would show that there are (infinitely many) non-vanishing Fourier coefficients
of $h_{\mu}$ for some primitive $\mu$ with summation index
\[\ell-\frac{1}{4}T^{-1}[\mu]= \frac{\alpha}{d_T},\]
(see \eqref{hmu-fe}) such that $\alpha $ is co-prime to $d$ (this is satisfied automatically by the primitiveness of $\mu$) and moreover that $\alpha$ square-free and odd.   

\subsubsection{Case $n$ is even}
Here we must multiply \eqref{start} by $2^n$ to get
\[d_{\mathcal T}= 4(\ell- \frac{1}{4}T^{-1}[\mu])\cdot d_T= 
(4 \ell-T^{-1}[\mu])\cdot d_T.\]

The middle factor has exact denominator $4d_T$ since $\mu$ is primitive (see \propref{prim}). In the next section ({\it Part~B}) we would show that there are (infinitely many) non-vanishing Fourier coefficients
of $h_{\mu}$ for some primitive $\mu$ with summation index (see \eqref{hmu-fe}) 
\[ \ell -\frac{1}{4}T^{-1}[\mu]= \frac{\alpha}{4d_T},\]
such that $\alpha $ is co-prime to $4d_T$ (this is satisfied automatically) and that $\alpha$ is odd and square-free.

The proposition below summarizes the findings from {\it Part~A} and also makes clear the role of {\it Part~B} in the remainder of the proof. Let us put
\begin{equation}\label{Hmu}
H_\mu(\tau) = \begin{cases} h_\mu(d_T \tau) \text{ if } n \text{ is odd, } \\  h_\mu(4d_T \tau) \text{ if } n \text{ is even}. \end{cases}
\end{equation}
We now apply the results of this section to $\varphi_T^{(r)}$ from \propref{sca-j} and keep in mind the Fourier expansion of $h_\mu$ from \eqref{hmu-fe} to arrive at the following.

\begin{prop} \label{parta}
Let $n \geq 2$ and $C \in \mf N$ be given. Assume that \thmref{mainthm} holds for
all non-zero forms in $M^{n-1}_\theta$ for all polynomial representations $\theta$ of $\mrm{GL}(n-1,\mf C)$ with $k(\theta)\ge C$. Consider a non-zero $F \in M^n_\rho$ where $\rho$ is any polynomial representation of $\glnc$ with $k(\rho)\ge C$.

Then there exist $T \in \Lambda_{n-1}^+$ with $d_T$ running over infinitely many odd, square-free numbers; and for each such $T$ there exist a primitive index $\mu \in \mf Z^{n-1} / (2T) \mf Z^{n-1}$ and an elliptic modular form $H_\mu(\tau) = \sum_{\ell \ge 1} a_\mu(\ell) q^{d'_T(\ell - \frac{1}{4} T^{-1}[\mu] ) }$ ($d'_T=d_T$ or $4d_T$ according as $n$ is odd or even) of weight at least $k(\rho) - (n-1)/2 $ as defined in \eqref{Hmu} with the following property:

$a_F(\mc T) \neq 0$ for $\mc T \in  \Lambda_{n}^+$ of the form $\mc T=\begin{psmallmatrix} \ell & \mu \\ \mu^t & T  \end{psmallmatrix}$ ($\ell \ge 1$) if $a_\mu(\ell) \neq 0$.

Such a $\mc T$ as above satisfies the property that $d_{\mc T}$ is odd and square-free provided one has that $d_T(\ell - \frac{1}{4} T^{-1}[\mu] )$ is odd and square-free.
\end{prop}
Therefore in {\it Part~B}, we must investigate the non-vanishing property of the Fourier coefficients $a_\mu(\ell)$. The reader would find the results summarized in \thmref{partb-c} and \thmref{partb-m}. Actually in our application of \propref{parta} to prove \thmref{mainthm} by induction on $n$ (see section~\ref{concl}), we would only need one such $T $ as in \propref{parta} for each of the induction steps. Say if we are at the $r$-th step (i.e., passing from $\mrm{Sp}(r)$ to $\mrm{Sp}(r+1)$, where $1 \le r \le n-1$), then we only need the non-vanishing of $\phi_T$ ($T \in \Lambda^+_r$) for one $T$; the statement about infinitely many
such $d_T$ follows from the corresponding property of the $a_{\mu}(\ell)$, where $\mu \in \mf Z^{r-1} / (2T) \mf Z^{r-1}$, $\ell \ge 1$.

\section{Part~B: The analytic part} \label{anaB}
We start with a non-zero $h_\mu$ with $\mu \in \mf Z^{n-1}/2T \mf Z^{n-1}$ being primitive (see \propref{parta}). The notion of $\mu$ being primitive with respect to $T$ can be found in \propref{prim}.
For the convenience of the reader, we remind that $h_\mu$ is a non-zero theta component of the scalar-valued Jacobi form $\varphi^{(r)}_T$ from section~\ref{jsca}, which in turn arises as a function component of the vector-valued Jacobi form $\varphi_T$. Moreover this $\varphi_T$ is a Fourier-Jacobi coefficient of $F$, see section~\ref{partas1}, especially \corref{nton1}. Note that $d_T$ is odd, square-free.

The arguments in this section are a little different according as $n$ is even or odd, but we try to treat them simultaneously. Put 
\begin{equation} \label{ddef}
\kappa = k- \frac{n-1}{2} , \q d = d_T,
\end{equation} 
and moreover for $N \geq 1$ and $\kappa \in \frac{1}{2} \mf Z$, we put (for the remainder of the paper)
\begin{equation} \label{n'}
 N'= \begin{cases} N \text{ if } \kappa \in \mf N, \\ 4N \text{ otherwise}. \end{cases} . 
\end{equation}
In any case we note that $h_\mu \in M_{\kappa}(\Gamma(d'))$ with $\kappa \in \mf Z$. Indeed, the transformation properties of the $h_\mu$ are inherited from those of the Jacobi form $\varphi^{(r)}_T$ and the Jacobi theta functions of matrix index (cf. \eqref{jtheta}). We can read these off from \cite[p.~210, eqn.~(1) and ~(2)]{Zi}. It is clear that the Weil representation defined in eg. \cite[Definition~5.1]{Str} is the same as that defined on the theta module $ \mf C \{ \Theta_T[\mu] \colon  \mu\in \mf Z^{n-1}/2T\cdot \mf Z^{n-1} \} $ (cf. section~\ref{jacobi} and and \cite[Lem.~3.2]{Zi}). Now the fact that $\phi^{(r)}_T = \sum_\mu h_\mu \cdot \Theta_T[\mu] $, implies that the tuple $ \{ h_\mu \}_\mu$ transforms via the dual of the Weil representation alluded to in the above. 
Thus our claim about the level of $h_\mu$ follows by noting that the kernel of the Weil representation attached to $2T$ factors through (and thus is trivial on) $\Gamma(d)$ when $n$ is odd (i.e., when $\kappa \in \mf Z$), and in general is trivial on $ \widetilde{\Gamma(\mrm{level} (2T))}= \widetilde{\Gamma(d') }$ (where $\widetilde{\Gamma}$ denotes the metaplectic cover of $\Gamma$, see e.g. \cite[Lem.~5.5]{Str} for this well-known result). For uniformity of notation, we suppress the tilde in $\widetilde{\Gamma}$ even when $\kappa \not \in \mf Z$.

Let us put $f=h_\mu(d' \tau)$. Therefore in \propref{parta} we take $H_\mu:=f$.

Then it is clear that $f \in M_\kappa(\Gamma_1(d'^2))$. Crucial for us would be the fact that the Fourier expansion of $f$ is supported away from its level (this follows from section~\ref{2adjust}, precisely because of the primitive-ness of $\mu$):
\begin{equation} \label{away}
f(\tau)=\ut{n \ge 1, \, (n,d')=1}\sum a_f(n) e(n \tau).
\end{equation}

For later purposes, we have to consider the modified cusp form
\begin{equation} \label{gdef}
g(\tau) := \ut{(n, \mf M)=1}\sum a_f(n) e(n \tau),
\end{equation}
where $\mf M$ (to be chosen later) is an odd, square-free integer containing all the prime factors of $d'$ if $\kappa$ is integral. We note that $g \in S_\kappa(\Gamma_1(d'^2 {\bf M'}^2))$, where $M'$ is the largest divisor of $\mf M$ coprime to $d'$. 

\begin{lem} \label{gn0}
 The modular form $g $ in \eqref{gdef} is non-zero. 
\end{lem} 
 
\begin{proof}
In the case of integral weights, this essentially follows since $f$ satisfies the property in \eqref{away} and follows easily from classical oldform theory. Let us put $\mf M':= \mf M/d'$. Then $g=0$ implies that $a(f,n)=0$ for all $n$ such that $(n, \mf M')=1$. Since $(\mf M', d'^2)=1$, from \cite[Theorem~4.6.8~(1)]{miyake} it then follows that $f=0$, a contradiction.
\end{proof}

\lemref{gn0} is true for half-integral weights as well, even though we will not use this fact. Since this is a bit subtle and may have use elsewhere, we give a proof. In this case let us put $\mf M$ to be the square-free number appearing in \cite[Thm.~3]{saha} (denoted by $N_f$ there). Recall that we need to show that $g$ (as in \eqref{gdef}) is non-zero.

We refer the reader to \cite[Prop.~3.7]{saha} for the spaces $S_\kappa(N,\chi)$ and follow its proof. Inspecting the argument in the proof of \cite[Prop.~3.7]{saha}, it is clear that one needs to prove that if $0 \neq {\bf f} \in M_\kappa(\Gamma_1(N'))$ for some $N \geq 1$ with $a_{{\bf f}}(n)=0$ for all $n$ such that $(n,p)=1$, then $p \mid N'$ (and hence $p \mid N$, since we can and would assume that $p$ is odd, since $N' $ and $N_f$ (cf. \cite[proof of Prop.~3.4]{saha}) are both even and here we are concerned with those primes $p$ which divide $N_f$ but not $N'$). It is not immediately clear how to adapt the proof in \cite[Lem.~7]{Se-St} in our setting. The lemma below however completes the proof.

\begin{lem} \label{oldhalf}
Let $\kappa \in \frac{1}{2} \mf Z$ and ${\bf f} \in M_\kappa(\Gamma_1(N'))$. If $a_{{\bf f }}(n)=0$ for all $n$ co-prime to an odd prime $p\nmid N'$; then ${\bf f}=0$.
\end{lem}

\begin{proof}
The trick is to reduce to integral weights. Suppose that ${\bf f} \neq 0$ and put ${\bf g}(\tau):= {\bf f}^2(\tau) $. Then from the formula
\[  a_{{\bf g}}(n) = \sum_{r+s=n} a_{{\bf f }}(r)a_{{\bf f }}(s) \]
we see that ${\bf g} \in M_{2\kappa}(\Gamma_1(N'))$ is such that $a_{{\bf g}}(n) =0$ for all $n$ with $(n,p)=1$. Indeed, each summand is zero unless both $r$ and $s$ are divisible by $p$. This means that the Fourier expansion of ${\bf g}$ is supported on multiples of $p$, and hence we can write $g(q) = \Psi(q^p)$ for some power series $\Psi$, where $q=e(\tau)$.

However this forces ${\bf g}$, and hence ${\bf f}$ to be zero upon invoking \cite[Chapter~VIII, Thm.~4.1]{Lang} which states that:

\noindent \textbullet {\it \, Let $L \in \mf N$ and $p$ be a prime such that $p \nmid L$. If $f_\infty(q)$ is a power series such that $f_\infty(q^p)$ belongs to $M_{j}(\Gamma_1(L))$ for some integer $j \ge 1$, then $f_\infty=0$;} 

by taking $f_\infty(q)= \Psi(q)$, $j = 2 \kappa$ and $L=N'$. 
This finishes the proof. \qedhere
\end{proof}

\subsection{{\bf Method~1:} \texorpdfstring{Proof for all cuspidal $f$ without using multiplicity-one for $\kappa \in \frac{1}{2} \mf Z$}{} } \label{ell-cusp}

In this section we want to prove that $f$ has infinitely many non-zero, odd and square-free Fourier coefficients, possibly in a quantitative fashion.
However, let us note that we can not just quote the corresponding results from (cf. \cite[Theorem~2]{AD} or \cite[Theorem~2]{saha}) since the results therein are only for cusp forms on $\Gamma_0(N)$ with nebentypus, whereas our setting is on $\Gamma_1(N)$, and the problem of finding non-zero square-free Fourier coefficients does not behave in a desirable way under the decomposition by characters.

We now pursue the cases of integral and half-integral weights in separate subsections, by closely following \cite{AD, saha}. We would henceforth work with the modular form $g$ from \eqref{gdef}. 

\subsubsection{\texorpdfstring{$f$ cuspidal and $\kappa$ integral}{}} \label{ell-icusp}

Let us put $D= d^2 {\bf M}^2$ and for a square-free $r$ such that $(r,D)=1$, define
\[ G:= U(r^2)g = \sum_{n \geq 1} a_g(r^2n)e(n \tau),\]  
so that $g \in S_\kappa(\Gamma_1(D))$, $G \in S_\kappa(\Gamma_1(Dr^2))$ (actually the level would be $Dr$ but we won't need this).

We apply the Rankin-Selberg method to $g$. For any ${\bf g} \in S_{\bf \kappa}(\Gamma_1(\bf D))$ (${\bf D} \ge 1$) with Fourier expansion $\mf g = \sum_{n \geq 1} a'_{\mf g}(n) n^{\frac{\kappa-1}{2}} e(n \tau)$; by applying this method get (see \cite[p.~357, Theorem~1]{Ran1}, also \cite[eq.~(1.14)]{Selb} and note that $a_{\mf g}(n)=a'_{\mf g}(n) n^{\frac{\kappa-1}{2}}$)
\begin{equation} \label{rs-f}
\underset{n\leq X}{\sum}|a'_{\mf g}(n)|^{2}=A_{\mf g}X+O_{\mf g}(X^{\frac{3}{5}}),
\end{equation}
where the constant $A_{\bf g}$ is given by
\begin{equation} \label{ag}
A_{\bf g} := \frac{3}{\pi}\frac{(4\pi)^{\bf \kappa}}{\Gamma(\bf \kappa)}\big[ \sltwo \colon \Gamma_1(\bf D)\big]^{-1}\left\langle {\bf g},{\bf g}\right\rangle_{\bf D}. 
\end{equation}
Here, and henceforth $\left\langle {\bf g},{\bf g} \right\rangle_{\bf D}$ denotes the Petersson norm of ${\bf g}$ with respect to $\Gamma_1(\bf D)$ defined by 
\[ 
\left\langle {\bf g},{\bf g} \right\rangle_{\bf D} := \int_{\Gamma_1(\bf D) \backslash \mf H } |\mf g(\tau)|^2 v^{k} dudv/v^2,  \qq (\tau=u+iv).
\]

It is possible to rework all the calculations done for this in \cite{AD} on the spaces $S_\kappa(N,\chi)$ in our situation, but since a major portion of the work requires only upper bounds on the sum of square of Fourier coefficients, we may reduce to \cite{AD} via decomposition by characters.

For $Y>0$ and a modular form $\mf g$, let us put $S_{\mf g}(Y) := \sum_{n \le Y} |a'_{\mf g}(n)|^2 $. We now prove some results which give upper and lower bounds for the quantity $S_{\mf g}(Y)$. In particular they provide suitable bounds on $S_{U(r^2)\mf g}((Y)$ which are uniform in $r$.

\begin{prop} \label{Mindep}
For $f \in S_\kappa(\Gamma_1(d^2))$ as above, the following statements hold for some positive constants $c_{f, \mf M}$ depending on $f$ and $\mf M$, and $B_f,C_f$ depending only on $f$.

(i) For all $Y \ge c_{f, \mf M}$ one has $S_g(Y)\ge   B_f  \prod_{p \mid \mf M} \left( 1 -  \frac{2}{p}   \right)^2 
 \cdot Y$,

(ii) For all $Y>0$ one has $S_g(Y)\leq C_{f} \cdot Y$.
\end{prop}

\begin{proof}
For the proof of $(i)$, let us write $\mf M_0:= \mf M/d$.
Since the Fourier expansion of $f$ is supported away from $d$, we see first of all
\begin{align}
S_g(Y)=\sum_{n \le Y, (n, \mf M)=1} |a'_{ f}(n)|^2  = \sum_{n \le Y, (n, \mf M_0)=1} |a'_{ f}(n)|^2 . \label{sglbd}
\end{align}
$\mf M_0$ is odd and square-free since both $\mf M$ and $d$ are so. Next, we rewrite \eqref{sglbd} as
\begin{align*}
S_g(Y) = \sum_{\beta \mid \mf M_0} \mu(\beta) \sum_{n \le Y/\beta}|a'_{ f}(n \beta)|^2  = \sum_{\beta \mid \mf M_0} \mu(\beta) S_{U(\beta)f}(Y/\beta).
\end{align*}
Therefore \eqref{rs-f} applied to $U(\beta)f \in S_\kappa(\Gamma_1(d^2 \beta))$ (with $\beta$ as above) gives
\begin{align} \label{sgasy1}
S_g(Y) &= \sum_{\beta \mid \mf M_0} \frac{\mu(\beta)}{\beta} A_{U(\beta)f} Y + O_{f,\mf M}(Y^{3/5}) .
\end{align}

For $\mf D \ge 1$, let us put $\nu_{\mf D}:= \big[ \sltwo \colon \Gamma_1(\mf D )\big]^{-1} $. Consider the orthogonal basis $\{ f_j\}_j$ away from $d^2$ that we are considering in this section.
Let us further recall from \cite[Theorem~8]{AD} and the argument in \cite[Proof of Cor.~5.4]{AD} that for $\beta$ and $f_j$ as above,
\begin{align} \label{ubreln}
\left\langle  U(\beta)f_j  , U(\beta)f_j \right\rangle_{ d^2 \beta} =  Q_\beta \left\langle f_j, f_j \right \rangle_{d^2 \beta} ,
\end{align} \label{qb}
where $Q_\beta$ is a multiplicative function given by 
\begin{equation}
Q_\beta(f_j) = \prod_{p \mid \beta} Q_p(f_j) ; \q Q_p(f_j) = \left( p^{k-2}+\frac{(p-1) |\lambda_{f_j}(p)|^2 } {p+1}  \right).
\end{equation}
Here $\lambda_{f_j}(p)$ is the eigenvalue for $f_j$ under $T_p$.
Now from \cite[Cor.~5.2]{AD} and \eqref{ubreln} we find that (with $a_{\kappa} =\frac{3}{\pi}   \frac{(4\pi)^{\bf \kappa}}{\Gamma(\bf \kappa)}$)
\begin{align*} 
A_{U(\beta)f_j} &= a_{\kappa}   \nu_{d^2 \beta} \beta^{1-\kappa} Q_\beta(f_j) \left\langle f_j, f_j \right \rangle_{ d^2 \beta} = a_{\kappa}   \nu_{d^2}  \beta^{1-\kappa} Q_\beta(f_j) \left\langle f_j, f_j \right \rangle_{ d^2 } = \beta^{1-\kappa} Q_\beta(f_j) \cdot A_{f_j}.
\end{align*}

Now we write $f = \sum_j c_j f_j$ and note that by the orthogonality of the $f_j$'s
\begin{align*}
\left\langle f, f \right \rangle = \sum_j |c_j|^2 \left\langle f_j, f_j \right \rangle ;  \q A_f = \sum_j |c_j|^2 A_{f_j} 
\end{align*}
and by the orthogonality of the $U_\beta(f_j)$'s (invoking \cite[Theorem~8]{AD}) that
\begin{align*}
\left\langle  U(\beta)f  , U(\beta)f \right\rangle_{ d^2 \beta} = \sum_j |c_j|^2 \left\langle  U(\beta)f_j  , U(\beta)f_j \right\rangle_{ d^2 \beta} .
\end{align*}
Therefore we get
\begin{align} \label{aureln}
A_{U(\beta)f}  = \sum_j |c_j|^2 \beta^{1-\kappa} Q_\beta(f_j) \cdot A_{f_j}.
\end{align} 

Putting all these together we now derive from \eqref{sgasy1} that
\begin{align} \label{sgpre}
S_g(Y) &= \sum_j |c_j|^2 A_{f_j} \sum_{\beta \mid \mf M_0} \frac{\mu(\beta) Q_\beta(f_j) }{\beta^\kappa}  \cdot  Y + O_{f,\mf M}(Y^{3/5}).
\end{align}
Let $Q_\beta$ (respectively $\lambda(p)$) denote any of the quantities $Q_\beta(f_j)$ (respectively $\lambda_{f_j}(p)$). 
We put $\mc Q(\mf M_0) := \sum_{\beta \mid \mf M_0} \frac{\mu(\beta) Q_\beta}{\beta^\kappa} $. By multiplicativity we can write
\begin{align*}
\mc Q(\mf M_0) = \prod_{p \mid \mf M_0} \left( 1 - \frac{Q_p}{p^\kappa} \right) 
= \prod_{p \mid \mf M_0} \left( 1 -  
  \frac{1}{p^2}  - \frac{(p-1) |\lambda(p)|^2 } {p^\kappa (p+1)}  \right) . 
\end{align*}
Using the Deligne's bound for $\lambda(p)$ we see that
\begin{align}
\mc Q(\mf M_0) &\ge \prod_{p \mid \mf M_0} \left( 1 -  \frac{4}{p} - \frac{1}{p^2} + \frac{8}{p(p+1)} \right) \n \\
& \ge \prod_{p \mid \mf M_0} \left( (1 -  \frac{2}{p})^2 +  \frac{8}{p(p+1)} - \frac{5}{p^2}  \right) \ge  \prod_{p \mid \mf M_0} \left( 1 -  \frac{2}{p}   \right)^2  \ge a_d \prod_{p \mid \mf M} \left( 1 -  \frac{2}{p}   \right)^2, \label{qm0}
\end{align}
for some constant $a_d$ depending only on $d$.

We now use the lower bound \eqref{qm0} in \eqref{sgpre} to obtain
\begin{align}
S_g(Y) &\ge  a_d\prod_{p \mid \mf M} \left( 1 -  \frac{2}{p}   \right)^2 \sum_j |c_j|^2 A_{f_j}  \cdot Y + O_{f,\mf M}(Y^{3/5}) \n \\
&\ge  A_fa_d \prod_{p \mid \mf M} \left( 1 -  \frac{2}{p}   \right)^2 
\cdot Y + O_{f, \mf M}(Y^{3/5}). \label{sgfinal}
\end{align}
The proof of $(i)$ is therefore complete from \eqref{sgfinal} taking for instance $B_f:= A_f a_d/2$.

For the proof of $(ii)$ note that for all $Y>0$,
\begin{align}
S_g(Y) = \sum_{n \le Y, (n, {\mf M})=1} |a'_{ f}(n)|^2 \le \sum_{n \le Y} |a'_{ f}(n)|^2  = S_f(Y) \le C_f \cdot Y,
\end{align}
for some constant $C_f$ depending only on $f$, by looking at \eqref{rs-f}.
\end{proof}

The space $S_\kappa(\Gamma_1(d^2))$ has an orthogonal basis consisting of eigenforms for all $T_n, \sigma_n$ with $(n,d)=1$. Here $\sigma_n$ are the diamond operators. This set is just the union of eigenforms away from the level $d^2$ in the spaces $S_\kappa(d^2,\psi)$ with $\psi$ varying $\mod{d^2}$. Let this basis be denoted by $\{f_1,f_2,\ldots, f_J\}$ where $J=\dim S_\kappa(\Gamma_1(d^2))$. 

\begin{lem} \label{flem}
For any square-free integer $r$ such that $(r, \mf M)=1$ and for all $Y>0$, one has $S_{U(r^2)f}(Y)\leq D_{f} \cdot 11^{\omega(r)}\cdot   Y$.
\end{lem}

\begin{proof}
We look at the orthogonal decomposition $f=\sum_j c_j f_j$, where the set $\{ f_1,\ldots, f_J \}$ is an orthogonal basis for $S_\kappa(\Gamma_1(d'^2))$.
We first prove the lemma for the eigenforms $f_j$. Let us put
$a'_j(n) := a'_{f_j}(n)$. Let $r=p_1p_2\cdots p_t$. We proceed by induction on $t$. Note that by our choice: cf. the paragraph preceeding \lemref{gn0}; $d \mid \mf M$. Thus $(r,\mf M) =1$ implies that $(r,d)=1$.

From \cite[Proof of Prop.~5.3]{AD} we recall that for a prime $p$ such that $(p, \mf M)=1$,
\begin{align}
a'_j(p^2n) = a'_j(n)a'_j(p^2) - \chi(p) a'_j(n) \delta_{p \mid n} -\chi(p) a'_j(n/p^2), 
\end{align}
where $\delta_{p \mid n} = 1$ if $p \mid n$ and $0$ otherwise and $a'_j(s)$ is zero if $s$ is not an integer. Here we remind the reader that the $|_k$ operator in \cite{AD} is normalised so that it defines a group action on $GL(2, \mf R)^+$, the subgroup of $GL(2, \mf R)$ whose elements have positive determinant. Since $f_j$ is an eigenform away from $d'$, $a'_j(p^2)$ equals the (normalised) eigenvalue of some newform of level dividing $D$ under the Hecke operator $T_{p^2}$ defined as in \cite{AD}. 

By the Cauchy-Schwartz inequality and applying the Deligne's bound to $a'_j(p^2)$, we get
\begin{align} \label{p1}
|a'_j(p^2n)|^2 \le 9 |a'_j(n)|^2 +  |a'_j(n)|^2 \delta_{p \mid n} + |a'_j(n/p^2)|^2 
 \le 10  |a'_j(n)|^2 + |a'_j(n/p^2)|^2 .
\end{align}
Summing \eqref{p1} over $n \le Y$ and using the bound in \propref{Mindep}~$(ii)$ applied to $g_j$, we get 
\[ S_{U(p^2)f_j}(Y) \le 11 \cdot S_{f_j}(Y) \le  11 \cdot C_j \cdot Y.\]
This proves \lemref{flem} for the eigenform $f_j$ when $t=1$ with a constant $C_{j}$ depending only on $\kappa$ and $d$. 

Let us put for $i=1,\ldots, t$, $V_i := U(p_1^2\cdots p_{i-1}^2)f_j$ and set $V_0 :=f_j$. By induction hypothesis, suppose we know the result in the statement of the lemma for $V_{i-1}$. Since $(p_{i},D)=1$, the Hecke operator $T_{p_i^2}$ commutes with $U(p_1^2\cdots p_{i-1}^2)$. Therefore arguing as in the case when $t=1$, replacing $g_j$ by $V_i$ we get
\begin{align} \label{vi}
S_{V_i}(Y) \le 11 S_{V_{i-1}} (Y) \le C_j \cdot 11^{i} \cdot  Y
\end{align}
with $C_j$ as above. This proves our result for the elements $\{f_1,\ldots f_J \}$ of the orthogonal basis. 

Now from the equation $f = \sum_j c_j f_j$ again, we see by the Cauchy-Schwartz inequality that
\begin{align}
S_{U(r^2) f}(Y) \le (\sum_j |c_j|^2) \cdot \sum_j S_{U(r^2 )f_j}(Y) \le D_{f} \cdot 11^{\omega(r)} \cdot  Y
\end{align}
using \eqref{vi} for $i=t$. Here $D_f= (\sum_j |c_j|^2)\cdot  (\sum_j C_j)$. This completes the proof of the lemma.
\end{proof}

Let $\mc S$ be the set of square-free integers, and $\mc S_{\bf M} = \{ n \in \mc S | (n, \bf M)=1 \}$. Define
\begin{equation} \label{Sgdef}
S_f({\bf M},X):= \sum_{n \le X, n \in \mc S_{\bf M} } |a'_f(n)|^2   \q (=\sum_{n \le X, n \in \mc S } |a'_g(n)|^2  )
\end{equation}
keeping in mind that $a'_g(n)$ is $0$ if $(n, {\bf M})>1$, and equal to $a'_f(n)$ otherwise, see \eqref{gdef}. Hence \footnote{We take the opportunity to correct an error in the proof of \cite[Prop.~5.8]{AD}. The calculations for the lower bound of the quantity $S_f(M,X)$ there are not correct, and should be replaced by those given below, along with the results supporting it, viz. \propref{Mindep} and \lemref{flem} of this paper. Rest of the results continue to hold in \cite{AD}. Note also that \propref{Mindep} and \lemref{flem} in this paper hold for any $\mf g$ whose Fourier expansion is supported away from its level.}
\begin{align}
S_f({\bf M},X)
= \sum_{n \le X} \sum_{r^2 \mid n} \mu(r) |a'_g(n)|^2
&= \sum_{r^2 \le X, r \in \mc S_{\bf M}} \mu(r) \sum_{m \le X/r^2} |a'_g(mr^2)|^2 \n \\
& =\sum_{r^2 \le X, r \in \mc S_{\bf M}} \mu(r) S_{U(r^2)g}(X/r^2).
\end{align}

Clearly $S_{U(r^2)g}(X/r^2) \le S_{U(r^2)f}(X/r^2)$.
Therefore for $X$ large enough (i.e. $X \ge c_{f, \mf M}$, where $c_{f, \mf M}$ is as in \propref{Mindep}), we can use $(i)$ and $(ii)$ in \propref{Mindep} and \lemref{flem} to write
\begin{align} 
S_f({\bf M},X) &\ge B_f   \prod_{p \mid \mf M} \left( 1 -  \frac{2}{p}   \right)^2 \cdot X- \sum_{r^2 \le X, r \in \mc S_{\bf M}}  S_{U(r^2)f}(X/r^2) \n \\
&\ge \left( B_f  \prod_{p \mid \mf M} \left( 1 -  \frac{2}{p}   \right)^2  - D_f  \cdot \sum_{r^2 \le X, r \in \mc S_{\bf M}}  \cdot 11^{\omega(r)} r^{-2} \right) \cdot X 
\label{lbd0}
\end{align}

Now let us choose ${\bf M}=\prod_{2<p \leq t} p$ with $t$ large enough such that all the primes in $d$ occur in ${\bf M}$ (cf. \eqref{ddef} and the few lines after \eqref{gdef}, and recall that $d$ is odd).
\begin{equation} \label{decay}
 \sum_{r \ge 2, r \in \mc S_{\mf M} } 11^{\omega(r)}/r^2 = -1 + \prod_{p>t}(1+11/p^2).
\end{equation} 
The latter sum is bounded as $O(1/t)$. Now we look at the quantity 
$\prod_{p \mid \mf M} \left( 1 -  \frac{2}{p}   \right)^2 $. We see that
\begin{align} \label{prodest}
\prod_{p \mid \mf M} \left( 1 -  \frac{2}{p}   \right)^{-2}  = \exp \left( 2 \sum_{2<p \le t}  \log(1- \frac{2}{p}) \right) =  (\log t)^4 +O(1) .
\end{align}
We therefore can choose $t$ large enough so that the expression inside the braces in \eqref{lbd0} is positive.  
Hence for $X \ge c_{f, \mf M}$ (which will depend only on $f$ since we can choose $t$ large depending only on $f$) that
\begin{equation} \label{lbd1}
S_f({\bf M},X) \ge E_f X
\end{equation}
where $E_f>0$ depends only on $f$.
Thus we are done in the integral weight case.

\subsubsection{$f$ cuspidal, $\kappa \in \frac{1}{2} \mf Z \setminus \mf Z$} \label{ell-hcusp}
Here as well, we try to reduce to the calculations in \cite{saha}. In the remainder of this subsection, let us put for convenience
\begin{displaymath}
\begin{aligned}
L:=d'^2, \q L_f := &\text{ the square-free integer divisible by all primes factors of } L \\ &  \, \text{ appearing in \cite[Thm.~3]{saha} (=$\mf M$, cf. the line preceeding \lemref{gn0}.)}
\end{aligned}
\end{displaymath}  

We start with
\begin{displaymath}
g(\tau) = \underset{n\leq X, \, (n,L_f)=1}{\sum} a'_f(n)q^n \in S_\kappa(\Gamma_1(LL_f^2)),
\end{displaymath}
here we do not worry about the precise level of $g$. Now $g $ is non-zero by \lemref{gn0}. 
Let us decompose $g$ by characters $\bmod{LL_f^2}$, and write
\begin{align*}
g = \sum_\chi c_\chi g_\chi,
\end{align*}
where $c_\chi \in \mf C$ and $g_\chi \in S_\kappa(LL_f^2,\chi)$. Note that the constants $c_\chi$ depend only on $g$ (or equivalently only on $f$). 
Let $M$ be a square-free integer (to be specified later) divisible by $L_f$.
Let us put $M_0=M/L_f$, so that $M_0$ is odd and square-free and $(M_0, L_f)=1$. We then consider the quantity
\begin{align*}
T_f(Y, M) &:= \underset{n \ge 1, \, (n,M)=1}{\sum}|a'(f,n)|^{2}e^{-n/Y} 
\end{align*}
which also equals $T_g(Y,M_0)$ whose definition is obvious. Then
\begin{align}
T_f(Y, M)=T_g(Y,M_0) &= \sum_{s \mid M_0} \mu(s) \sum_{n \ge 1} |a'(g,sn)|^{2}e^{-sn/Y} \n \\
&= \sum_{\chi,\chi'} c_\chi \overline{c_{\chi'} }\sum_{s \mid M_0} \mu(s) \sum_{n \ge 1} a'(g_\chi,sn) \overline{a'(g_{\chi'},sn) } e^{-sn/Y}. \label{chi1}
\end{align}
If we denote the innermost sum in \eqref{chi1} by $T_s(\chi,\chi';Y)$, then from the results of \cite{DI} one infers that for any square-free $s$ such that $(s,L_f)=1$ (in particular for $s \mid M_0$), and for some constant $\mc C_f>0$ (proportional to $\langle f, f \rangle$) depending only on $f$
\begin{align}
T_s(\chi,\chi;Y) = \mc C_f \frac{Y}{s} + O_{f,M}(Y^{1/2}) \label{chichi}\\ 
T_s(\chi,\chi' ;Y) = O_{f,M}( Y^{1/2}) \q \q (\chi \neq \chi')\label{chichi'}.
\end{align}
In \eqref{chichi'}, we have used the fact that $g_\chi$ and $g_{\chi'}$ are orthogonal to each other if $\chi \neq \chi'$. Indeed \eqref{chichi'} follows from \cite[Theorem~5]{DI} noting that the first term on the r.h.s. of \cite[eqn~(48)]{DI} is proportional to the inner product of the cusp forms in question.
Therefore 
\begin{align}
T_g(Y,M_0) &= \sum_{\chi} |c_\chi |^2 \sum_{s \mid M_0}  \frac{\mu(s)}{s} \cdot \mc C_f Y + O_{f,M}( Y^{1/2}) \n \\
& \ge \mc D_f \frac{\phi(M)}{M} \cdot Y, \label{pass}
\end{align}
for all $Y>d_{f,M}$ for some constant depending only on $f,M$.

After this, we follow the argument in the previous subsection, i.e., choose an orthogonal basis of $S_\kappa(\Gamma_1(L))$ and write $f = \sum_{j} f_j$, where the set $\{f_j\}_j$ consists of a basis of pairwise orthogonal eigenforms on the spaces $S_\kappa(L,\chi)$ ($\chi \bmod L$) which are away from $L$. We obtain analogous to \cite[Lem.~3.8]{saha} that for any square-free $r$ (including $1$) co-prime with $L$
\begin{equation} \label{bf}
\sum_{n \ge 1} |a'_{U(r^2)f}(n)|^{2} e^{-n/X} \le 19^{\omega(r)} B_f  X,
\end{equation}
where $B_f$ depends only on $f$.

Let us now finish the proof. Recall that $\mc S_M$ is the set of square-free integers coprime to $M$. We write using \eqref{pass} and \eqref{bf} that

\begin{align} \label{halflbd}
\underset{n\ge 1, \, n \in \mc S_{M }}{\sum}|a'_f(n)|^{2} e^{-n/X}\ge \frac{\phi(M)}{M} D_f \cdot X - B_f  X \cdot \sum_{r \ge 2, (r,M)=1} |\mu(r)| 19^{\omega(r)} r^{-2}.
\end{align}
The right hand side is shown to be $\gg X$ in \cite[p.~377]{saha} by choosing $M$ appropriately (depending only on $f$, see also subsection~\ref{ell-icusp}), and thus we are done.

\subsubsection{Quanitative bounds} \label{quanta}
To obtain quantitative versions of the non-vanishing results, we use Deligne's bound ($ a'_f(n) \ll_f n^{\epsilon} $, $n \ge 1$ square-free) when $\kappa \in \mf Z$ and Bykovski\u{i}'s bound ($a'_f(n) \ll_f n^{ 3/16 +\epsilon} $, $n \ge 1$) \cite{Byko} otherwise. More precisely, when the weight $\kappa$ is integral, referring to the sum in \eqref{Sgdef} and looking at its lower bound from \eqref{lbd1}, we see (with a given $\epsilon>0$) that
\begin{equation} \label{ilbd}
 \# \{ n \le X, n \in \mc S_{\bf M} | a'_f(n) \ne 0 \} \cdot X^{\epsilon}  \gg \sum_{n \le X, n \in \mc S_{\bf M} } |a'_f(n)|^2 \gg_f X,
\end{equation}
which gives the lower bound on $\# \{ n \le X | a'_f(n) \ne 0 \}$.
When $\kappa$ is half-integral, we start from the lower bound in \eqref{halflbd} writing it in the form
\begin{equation} \label{hlbd1}
\underset{n\le X^{1+\epsilon}, \, (n,M)=1}{\sum}|a'_f(n)|^{2} e^{-n/X} + \underset{n > X^{1+\epsilon}, \, (n,M)=1}{\sum}|a'_f(n)|^{2} e^{-n/X}\gg_f X.
\end{equation}
The second sum above can be estimated as follows. 
\begin{align} 
\underset{n > X^{1+\epsilon}, \, (n,M)=1}{\sum}|a'_f(n)|^{2} e^{-n/X} 
& \le \int_{X^{1+\epsilon}}^\infty  y^{3/8+\epsilon} e^{-y/X}dy \n \\
& \le X^{11/8+\epsilon} \Gamma(3/8+\epsilon, X^{\epsilon}) \label{hlbd2}
\end{align}
where $\Gamma(s,x)$ denotes the upper incomplete Gamma function defined for $\Re(s)>0$ as:
\[  \Gamma(s,x) = \int_x^\infty y^{s-1} e^{-y} dy. \]
From the standard asymptotic properties of $\Gamma(s,x) $ (see eg. \cite{abra}) for $s>0$ and $x \to \infty$ we know that $\Gamma(s,x) \sim x^{s-1}e^{-x}$. Therefore for $X$ large enough we see that the left hand side of \eqref{hlbd2} is
\begin{equation*} 
\ll_\epsilon X^{11/8+\epsilon} e^{-X^\epsilon } \ll_{\epsilon,A} X^{-A}
\end{equation*}
for any $A>0$. Therefore the second sum in \eqref{hlbd1} contributes to it negligibly.

Thus for large enough $X$ and the obvious inequality:
\[  \underset{n\le X^{1+\epsilon}, \, (n,M)=1}{\sum}|a'_f(n)|^{2} \geq \underset{n\le X^{1+\epsilon}, \, (n,M)=1}{\sum}|a'_f(n)|^{2} e^{-n/X} \gg_{f,\epsilon} X, \]
we get using Bykovski\u{i}'s bound mentioned earlier in this section that
\begin{equation*} 
 \# \{ n \le X^{1+\epsilon}, n \in \mc S_{ M} | a'_f(n) \ne 0 \} \cdot X^{3/8+\epsilon}  \gg \sum_{n \le X^{1+\epsilon}, (n,  M)=1 } |a'_f(n)|^2 \gg_f X.
\end{equation*}
Finally by replacing $X^{1+\epsilon}$ by $X$ and again changing $\epsilon$ if necessary we obtain
\begin{equation}\label{hlbd3}
\# \{ n \le X, n \in \mc S_{ M} | a'_f(n) \ne 0 \} \gg_{f, \epsilon} X^{5/8 - \epsilon}. 
\end{equation}
We summarize the above in a theorem.

\begin{thm} \label{partb-c}
Let $\kappa \geq 2, N \ge 1$ and $ f \in S_{\kappa}(\Gamma_1(N'))$ ($N'$ as defined in \eqref{n'}) be such that $a_{ f}(n)=0$ for all $n$ such that $(n,N')>1$. Then there exist infinitely many odd and square-free $n$ such that $a_{ f}(n) \neq 0$. More precisely, for any $\epsilon>0$,
\[ \# \{ n \le X, \, n \text{ square-free } |\,  a_{ f}(n) \neq 0 \} \gg_{f, \epsilon} \begin{cases} X^{1-\epsilon} &\text{ if } \kappa \in \mf Z  \\ X^{5/8 - \epsilon} &\text{ otherwise} \end{cases}. \]
\end{thm}

\subsection{{\bf Method~2:} \texorpdfstring{Proof for all $f \in \mN$ using multiplicity-one where $\kappa \in \mf Z$}{} } \label{ell-ncusp}

We would show that when $\kappa \in \mf Z$, one can prove an optimal result about the nonvanishing of odd, square-free Fourier coefficients of any non-zero $h \in \mN$ whose Fourier expansion is supported away from the level. One could then apply this result to $f$ from \eqref{away} with $N=d^2$. The crucial input would the multiplicity-one on the space of newforms. The method is based on an useful argument involving sieving of eigenforms, originally due to Balog-Ono \cite{bo}. Variants of this method have already appeared in \cite{AD} and \cite{BD3}. For convenience, we recall the details of the method below.

We start with an $h \in \mN$ such that $a_h(n)=0$ for all $n$ such that $(n,N)>1$. We consider, as before, $H(\tau)=h(N \tau) \in M_\kappa(\Gamma_1(D))$ where we put $D=N^2$. Note that $a_H(n)=0$ for all $(n,D)>1$.

Let us denote the set of normalised newforms on $\Gamma_1(D)$ of weight $k$ by $\mc F^1_D$ and those on $\Gamma_0(D)$ with nebentypus $\chi$ by $\mc F_{D,\chi}$. \textsl{In this section we allow $\kappa \ge 2$}.

Since $\mc F^1_D = \cup_\chi \mc F_{D,\chi}$ where $\chi$ runs over Dirichlet characters $\bmod D$; we can write
\begin{equation} \label{fexpan}
H(\tau)=  \sum_i \sum_\delta \alpha_{i,\delta} \, f_{i}(\delta \tau).
\end{equation}
Here $f_i$ runs over $\cup_\chi \mc F_{D,\chi}$, $\delta$ runs over the divisors of $D$ (such that $\delta \cdot \text{(level of } f_i) \mid D$); and $c_{j,\delta} $ are scalars. For this result, see \cite[Lem.~4.6.9]{miyake} for the space of cusp forms and \cite[Prop.~5]{weis} for the space of Eisenstein series.

Now by classical (new)old-form theory it follows that $H$ can not entirely lie in the old-space (see \cite[Theorem~4.6.8]{miyake}), and hence must have a non-zero `newform' component of some level. In terms of \eqref{fexpan}, this means that there exist an $i'$ such that $c_{i',1} \neq 0$. For ease of notation, by renumbering if necessary we can assume that $i'=1$ and omit the nebentypus $\chi'$ of the corresponding newform $f_1$ (of some level $M \mid D$) from display.

We would now pass on to the argument mentioned at the beginning of this subsection. This involves sieving out the entire old-classes corresponding to all $f_i$ with $i \neq 1$ by means of Hecke operators $ T_p$ with $p \nmid D$ (for $\Gamma_1(D)$) and using multiplicity-one. Let the Fourier expansion of $f_i$ be written as
\begin{equation*}
f_i(\tau) = \sum_{n \geq 1} b_i(n) q^n.
\end{equation*}

By multiplicity-one, we can choose an odd prime $q=q_{1,2} \nmid D$ 
such that $b_{1}(q) \neq b_{2}(q) $. In fact, one can bound such a $q$ solely in terms of $\kappa$ and $D$. See eg. \cite[Prop.~3.2]{BD3}.
Then consider the form $g_{1}(\tau)=\sum_{n=1}^{\infty}a_{1}(n)q^{n} :=T(q) f(\tau)-b_{2}(q)f(\tau)$ so that
\begin{align*}
g_{1}(\tau)=\sum_{i=1}^{s}(b_{i}(q)-b_{2}(q))\sum_{\delta|N} \alpha_{i,\delta}f_{i}(\delta \tau) \q (s=\dim M_\kappa(\Gamma_1(D)) \, ).
\end{align*}
The modular forms $f_{2}(\delta \tau)$ for any $\delta | N$, do not appear in the decomposition of $g_{1}(\tau)$ but $f_{1}$ does. Proceeding inductively in this way, we can remove all the non-zero newform components and their translates $f_{i}(\delta \tau)$ for all $i=2,...,s$ one by one, to obtain a modular form $F$ in $M_\kappa(\Gamma_1(D))$ such that on the one hand, we have
\begin{equation} \label{gs-1}
F(\tau)=\sum_{n=1}^{\infty}A(n)q^{n}:=\prod_{2 \le j \le s}  (b_{1}(q_{1,j})-b_{j}(q_{1,j}))    \sum_{\delta|N}\alpha_{1,\delta}f_{1}(\delta \tau).
\end{equation}
To be precise, we put $g_0:=H$, and having determined $g_{i-1}$, we put $g_i := T(q_{1,i})g_i - b_{q_{1,i}} g_i$, where $q_{1,i}$ is any prime coprime to $D$ such that $b_{1}(q_{1,i}) \neq b_{i}(q_{1,i}) $. At each stage we choose the smallest odd prime with these properties. As mentioned before, such primes are bounded solely in terms of $\kappa$ and $D$ (cf. \cite[Prop.~3.2]{BD3}). Then $F=g_{s-1}$.

By the construction, the product appearing on the right hand side of the equality in \eqref{gs-1} is non-zero. Therefore rescaling $F$, and calling the resulting function again as $F$, we on the other hand note that the inductive procedure gives us finitely many ($\le 3^{s-1}$) algebraic numbers $\beta_{t}$ (polynomials in the $p_{1,i}$'s and certain Dirichlet characters) and positive, square-free rational numbers $\gamma_{t}$ (which are quotients of the $p_{1,i}$'s) such that for every $n \ge 1$
\begin{equation} \label{f}
A(n) =\sum_{\delta \mid N} \alpha_{1,\delta} b_{1}(n/\delta) \equiv \sum_{t} \beta_{t}a_H(\gamma_{t}n) 
\end{equation}
In \eqref{f}, $\alpha_{1,\delta} \neq 0$. We put
$Q:=  \prod_{i=2}^s q_{1,i}$.
Note that $\max_{t} \gamma_t \le Q$.

Therefore if we choose $n \ge 1$ such that $(n,QN)=1$, we get
\begin{equation} \label{neweq}
\alpha_{1,1}b_1( n) = \sum_t \beta_{t}a_H(\gamma_{t}  n) .
\end{equation}
Then by our choice, all the $\gamma_t$ appearing in \eqref{neweq} are odd, square-free integers. This is because the primes $q_{1,i}$ are pairwise distinct. 

Let us define for $L \ge1$ and $\mf g \in M_\kappa(\Gamma_1(D))$ the counting functions:
\begin{align*}
\Pi_{\mf g}(L;X) &:=\{ n \le X \mid (n,L)=1, \, a_{\mf g}(n) \neq 0 \} \\
\Pi_{\mf g}^* (L;X) &:= \# \{ n \le X \mid (n,L)=1, n \text{ odd, square-free, } \, a_{\mf g}(n) \neq 0 \}.
\end{align*}

Now from \eqref{neweq} we see that for any $n \le X$, $(n,QN)=1$ such that $b_1(n) \neq 0$, there is a \textsl{unique} $t$ such that $\gamma_tn$ satisfies $a_H(\gamma_t n) \neq 0$. Clearly $\gamma_t n \le QX$.
This implies that
\begin{equation} \label{piineq}
\Pi^*_{f_1}(QN;X) \leq \Pi^*_H(QN, QX).
\end{equation}

The size of the set $\Pi_{\mf g}(L;X)$ has been studied in \cite{serre1, serre2} when $\mf g$ is a newform. In particular when $\mf g$ is cuspidal, we quote \cite[Th\'{e}or\'{e}me~16, Prop.~18]{serre1} for $\kappa \ge 2$ and \cite[Th\'{e}or\'{e}me~4.2~(ii)]{serre2} for $\kappa=1$. For our purposes we need analogous statements about the quantity $\Pi^*_{\mf g}(L;X)$. Namely we want to show that for some constants $\rho_j , \alpha >0$ depending only on $\mf g$ and $L$ the following hold:
\begin{align} \label{piasymp}
\# \Pi^*_{\mf g}(L;X) \sim \begin{cases} \rho_1 X &\text{ if } \mf g \text{ is cuspidal, not of CM type, } \kappa \geq 2, \\
  \rho_2 \frac{X}{(\log X)^{1/2}} &\text{ if } \mf g \text{ is cuspidal, of CM type, } \kappa \geq 2,  \\
  \rho_3 \frac{X}{(\log X)^{\alpha}} &\text{ if }  \mf g \text{ is cuspidal, } \kappa =1, \\
 \rho_4 X &\text{ if } \mf g \text{ is an Eisenstein newform, } \kappa \ge 2.   \end{cases}
\end{align}

These bounds have been proved in \cite{serre1, serre2} for $\Pi_{\mf g}(L;X)$. We need to address two additional points in our case, namely: 

(a) Serre states his results for $L=1$, but we require them to hold for any fixed $L \ge 1$,

(b) we have to count odd, square-free integers, i.e., consider the quantity $\Pi^*_{\mf g}(L;X)$.

We show below how to adapt Serre's arguments in a simple manner to deal with (a) and (b). For this we refer the reader to \cite[\S~1, \S~2]{serre2}, and follow the arguments presented therein.

Let $P_{\mf g}(L) := \{ p \nmid L \, \mid a_{\mf g}(p) = 0 \}$ and $\Pi^*_{\mf g}(L) := \{ (n, L)=1, n \text{ square-free} \, \mid a_{\mf g}(n) = 0 \}$.

Then the generating function (Dirichlet series) $\mc F(s)$ of $ \Pi^*_{\mf g}(L) $ is just 
\[ 
\mc F(s)  = \sum_{n \not \in \Pi_{\mf g}^* (L)} n^{-s}  = \prod_{p \not \in  P_{\mf g}(L) } (1+p^{-s}) = \sum_{n \ge 1} b_n n^{-s}, \, \text{  say}.
\] 
Clearly $\Pi^*_{\mf g}(L;X) = \sum_{n \le X} b_n$. Further, the set $P_{\mf g}(L) $ is "Frobenian" and has natural (and hence analytic or Dirichlet) density $\alpha :=\alpha(\mf g)$ such that $0 \leq \alpha <1$. Moreover if $\mf g$ has CM., then $\alpha=1/2$ and if $\kappa \ge 2$ and $\mf g$ doesnot have CM., then $\alpha=0$. For these facts see e.g. \cite[\S~7.4, p.~178]{serre1} and \cite[\S~7.5, p.~180]{serre1} respectively.

$\mc F(s)$ is holomorphic in the region $\Re(s)>1$ and is non-zero there.
We can then write
\[ \log \mc F(s) = \sum_{p \not \in  P_{\mf g}(L)}  p^{-s} + \theta_1(s) = (1-\alpha)\cdot \frac{1}{s-1} + \theta_2(s); \]
where $\theta_1(s), \theta_2(s)$ are holomorphic in $\Re(s) \ge 1$. This can be seen e.g. from \cite[(1.6)]{serre2}. This gives us that
\[ \mc F(s) = \frac{1}{(s-1)^{1-\alpha}} \cdot \exp (\theta_2(s)), \q\q (\Re(s) \ge 1). \]
Since $\mc F(s)$ has non-negative Dirichlet coefficients, our desired results follow from the (generalised) Ikehara-Weiner theorem \cite[(2.7)]{serre2}, \cite{D} and we conclude that
\[ \Pi^*_{\mf g}(L;X) = \sum_{n \le X} b_n \sim c \cdot \frac{X}{(\log X)^\alpha}, \q\q (c>0). \]
The assertions in \eqref{piasymp} now follow from the different values of $\alpha$.

From \eqref{piineq}, and from the discussion above, we now arrive at the following theorem.

\begin{thm} \label{partb-m}
Let $\kappa \in \mf Z$, $N \ge 1$ and $ h \in M_{\kappa}(\Gamma(N))$ be such that $a_{ f}(n)=0$ for all $n$ such that $(n,N)>1$. Then there exist infinitely many odd and square-free $n$ such that $a_{ f}(n) \neq 0$. More precisely, for some $0<\alpha<1$,
\[ \# \{ n \le X, \, n \text{ odd, square-free } |\,  a_{ h}(n) \neq 0 \} \gg_{h} \begin{cases}  
\frac{X}{ (\log X)^{1/2}  } &\text{ if } h \text{ is cuspidal, } \kappa \ge 2,\\ 
\frac{X}{(\log X)^{\alpha}} &\text{ if }  h \text{ is cuspidal, }  \kappa =1,\\
X &\text{ if } h \text{ is not cuspidal, } \kappa \geq 2. \end{cases}
 \]
\end{thm}

\subsection{Conclusion of the proof of \thmref{mainthm}} \label{concl} 
As mentioned in the introduction, we prove \thmref{mainthm} by induction on $n$. When $n=1$ and $\rho$ is as in the statement of the theorem, we can write $F=(f_1,\ldots, f_m)$, where $m=\dim \rho$; and each $f_j$ is a modular form of some weight $k_j \ge 0$. Moreover at least one of the $f_j \neq 0$. To prove \thmref{mainthm} for $n=1$, it is therefore enough to prove \thmref{mainthm} for any of the non-zero $f_j$. This in turn follows directly from \thmref{partb-m}. 

We next move to treat higher degrees and assume that $n>1$.

We first demonstrate the proof of the first two lower bounds of $\# \mk S_F(X)$ in \thmref{mainthm}.
We start with a {\bf \textsl{vector-valued}} $F\ne 0$ of degree $n$ with the given condition on the weight that $ k(\rho) - n/2 \ge \varrho(n)$ and consider the non-zero modular form $F^o \in M^{n-1}_{\rho''}$ (with $\rho''$ denoting the representation appearing in \propref{vec}) of weight $k(\rho'') \geq k(\rho)$. To apply the induction hypothesis to $F^o$ we need to verify that $k(\rho'') - (n-1)/2 \ge \varrho(n-1)$. Indeed this follows from the above inequality and the hypothesis on $F$. Now using \corref{nton1}, we get hold of a scalar-valued Jacobi form $\varphi^{(r)}_T \ne0 $ ($T \in \Lambda^+_{n-1}$) which is a Jacobi form of weight $ k' \ge k(\rho)$ and is a certain vector component of a (vector-valued) Fourier-Jacobi coefficient $\varphi_T$ of $F$. Here $T$ has odd, square-free discriminant. 

Now we again invoke the fact that $k(\rho) - \frac{n}{2} \ge \varrho(n)$, so that $\kappa := k' - (n-1)/2 \ge k(\rho) - (n-1)/2 \ge \varrho(n)+1/2$ and hence that $\kappa \geq 5/2$ since $\kappa$ is half-integral when $n$ is even. This allows us to use the results of section~\ref{anaB}.  If $n$ is odd, we get $\kappa \geq 2$ by a similar reasoning, this puts us into the setting of section~\ref{ell-ncusp}.
See \rmkref{wts} for more on the restriction on the weights.

Then we pass on to the realm of elliptic modular forms by using \propref{parta}, by getting hold of a primitive theta component $h_\mu$ of $\varphi^{(r)}_T$ of weight $\kappa$. In fact we work with its close relative $H_\mu$ (cf. \eqref{Hmu}) and then use \thmref{partb-c} and \thmref{partb-m} to get suitable nonvanishing properties of its Fourier coefficients. In the vector-valued setting that we are in, let us point out that even when we start with a \textsl{non-cuspidal} form $F$, it is not clear to us how ensure that the scalar-valued $\varphi^{(r)}_T$ is also non-cuspidal. This is highly probable, but we cannot prove it. One can prove this however if we start from a non-cuspidal scalar-valued modular form, see below. Therefore the first lower bound in \thmref{mainthm} is actually the infimum of the lower bounds appearing in \thmref{partb-m}.

Thus we can demonstrate (via \propref{parta}) the requisite nonvanishing properties of the Fourier coefficients of $F$. This finishes the proof of the first two lower bounds of $\# \mk S_F(X)$ in \thmref{mainthm}.

To demonstrate the proof of the last lower bound in \thmref{mainthm}, we start with a {\bf \textsl{scalar-valued}} $F \neq 0$. We then consider $G:= \Phi(F) \in M^{n-1}_k$, where $\Phi$ is the Siegel's $\Phi$-operator (see \eqref{phiop}). Now the first lower bound in \thmref{mainthm} gives us at least one $M \in \Lambda^+_{n-1}$ such that $d_M$ is odd, square-free and,
\begin{equation} \label{phincusp}
a_{\Phi(F)}(M) = a_F (\begin{psmallmatrix} M & 0 \\ 0 & 0     \end{psmallmatrix}) \neq 0.
\end{equation}
If the Fourier Jacobi coefficients of $F$ are denoted by $\varphi_T$ ($T \in \Lambda_{n-1}$), then \eqref{phincusp} clearly implies that $\varphi_M \neq 0$ and that $\varphi_M$ is non-cuspidal. Therefore by \lemref{jncusp} and \propref{primh}, $\varphi_M$ has a non-zero \textsl{non-cuspidal} primitive theta component $h_\mu$. From this point on the demonstration proceeds exactly as described in the first part of this proof. We apply \thmref{partb-m} here. The restriction on the weight comes from \lemref{jncusp}. 
\QEDB

\begin{rmk}[Condition on weights] \label{wts}
Let us remark here that the condition on the weight, namely $ k(\rho) - \frac{n}{2} \ge \varrho(n)$ in \thmref{mainthm} is technical, and is used at many places in section~\ref{anaB}. Let us put $\ell:= k(\rho) -\frac{n-1}{2}$. When $n$ is odd, so that $\ell$ is an integer, the bound $\ell \ge 2$ is enough. When $n$ is even, so that $\ell$ is a half-integer, the bound $\ell \ge 5/2$ is needed in order to invoke results from \cite{saha}. 

To be more precise let us point out that the results of subsections~\ref{ell-icusp} and~\ref{ell-hcusp} are valid for $\kappa $ (= weight of the cusp form considered there) at least $1$ and $5/2$ respectively. The lower bound $5/2$ can be improved to $3/2$ if one uses the results of \cite{yli}, however one has to ensure that we do not encounter unary theta series of weight $3/2$. This is an interesting thing to consider. Whereas in section~\ref{ell-ncusp} we request $\kappa \ge 1$ in the cuspidal case and $\kappa \geq 2$ otherwise. The second condition above perhaps can be relaxed.

We note here that in the scalar-valued {\it cuspidal} case it is necessary that $k - \frac{n}{2} \ge 0$, otherwise $F$ is singular. Thus our condition on $k$ implies that $F$ is not singular. Moreover our result is false for the small weights: eg. when $k=n/2 +1$. Counterexamples are furnished by the theta series $\vartheta \in S^n_{n/2+1}$ given as
\[   
\vartheta (Z) = \sum_{X \in M_n(\mf Z)} \det(X) e(S[X] Z) 
\]
where $S$ is even unimodular and does not have an automorphism with $\det = -1$. In particular it is not clear whether our theorem would hold for $k = (n+1)/2$. Similar remarks as above apply to the vector-valued case as well.

In the {\it non-cuspidal} case as well our theorems may not hold for small weights; e.g., the classical theta function of weight $n/2$ defined as $\vartheta$ above, but with the $\det(X)$ removed, is a counterexample.
\end{rmk}

\section{\texorpdfstring{Refinement for prime disciminants and an application to the spinor $L$-function}{} } \label{app}
In this section we first show how a variant of our main result (\thmref{mainthm}) can be used to refine it to a statement about ``\textsl{prime discriminants}". This is explicitly stated below. Let $\mf P$ denote the set of all primes.

\begin{thm} \label{oddthm}
Let $n$ be odd. Let $F \in S^n_\rho$ be non-zero and $k(\rho) - \frac{n}{2} \ge 3/2 $. Then there exist $T \in \Lambda_n^+$ with $d_T$ assuming infinitely many odd prime values, such that $a_F(T)\neq 0$. Moreover, the following stronger quantitative result holds:
\begin{align*}
\# \left( {\mk S}_F (X) \cap \mf P \right) \gg X/\log X,
\end{align*}
where the implied constant depends only on $F$.
\end{thm}

We now show how to use the above theorem (for $n=3$) along with the work of A. Pollack \cite{Pol} to obtain the standard analytic properties (meromorphic continuation, functional equation etc.) of the spinor $L$-function $Z_F(s)$ of a holomorphic Siegel cuspidal eigenform $F$ on $\spt$ unconditionally (cf. \thmref{poll} which was stated in the Introduction). We briefly discuss the background behind this result. Pollack used the correspondence between ternary quadratic forms and quaternion algebras to study a certain Rankin-Selberg integral (with respect to a suitable Eisenstein series) indexed by orders in quaternion algebras (or equivalently by some $T \in \Lambda_3^+$). This integral could be evaluated by unfolding using an expression for the spinor $L$-function as a Dirichlet series (essentially due to Evdokimov \cite{evdo}); with a factor $a_F(T)$ in the front: where $T$ corresponds to a maximal order in the quaternion algebra in question. The moment one knows that $a_F(T) \ne 0$, one can read off the analytic properties of $Z_F(s)$ from those of the Eisenstein series in question. This is what we are going to do in this section. But first, let us postpone the proof of \thmref{oddthm} and show how one can obtain \thmref{poll} from it.

\subsection{Proof of \thmref{poll}}
In view of \cite[Theorem~1.2]{Pol} and \thmref{oddthm}, it is enough to check that $T \in \Lambda^+_3$ with $d_T=p$ ($p$ odd prime) defines a maximal order in a quaternion algebra (necessarily) ramified at $\infty$ since $T$ is positive-definite.

Since the correspondence between $\Lambda^+_3$ and orders in quaternion algebras (see e.g. \cite[Proposition~3.3]{Pol} or \cite[Chapter~22]{voi}) preserves discriminants (cf. \cite[Corollary~3.4]{Pol}); if $T$ corresponds to an order $\mc O_T$ in some quaternion algebra $Q$ over $\mbb Q$ ramified at $\infty$, then $d_T=p=|rd(\mc O_T)|$. Here $rd(\mc O_T)$ denotes the reduced discriminant of $\mc O_T$. This implies that $\mc O_T$ is a maximal order. Indeed, if $\mc O_T \subset \mc O$ for some \textsl{maximal} order $\mc O$ of $Q$, then $|rd(\mc O)| \big{|} |rd(\mc O_T)|=p$. However $rd(\mc O)$ can not be $1$ as it must be ramified at another finite place (viz. $p$) since the number of ramified places is even. Thus $|rd(\mc O)|=|rd(\mc O_T)|$, i.e., $\mc O = \mc O_T$ and $\mc O_T$ is maximal. \qed

In order to prove \thmref{oddthm}, we need a lemma.

Let us denote the set of normalised newforms on $\Gamma_1(M)$ of weight $k$ by $\mc F^1_M$ and those on $\Gamma_0(M)$ with nebentypus $\chi$ (with $m_\chi \mid M$) by $\mc F_{M,\chi}$. We begin with a lemma on integral weight cusp forms. For $\ell \ge 1$, $f \in S_k(\Gamma_1(N))$, let $f \mid_k V_\ell (\tau):= f(\ell \tau)$. In this section we allow $k \ge 1$.

\begin{lem} \label{appln}
Let $f \in S_k(\Gamma_1(N))$ be non-zero. Suppose that $f$ does not belong to the $\mbb C$-span of $\mc F^1_M \mid_k V_\ell$ (where $M \mid N$ and $\ell \mid N/M$) with $\ell>1$. Then there exist infinitely many primes $p$ such that $a_f(p) \ne 0$. More precisely, 
\[  \# \{ p \le x, \, (p,N)=1 \mid a_f(p) \ne 0  \} \gg_{f} x/\log x.  \]
\end{lem}
The proof of the above lemma is based on the celebrated Ikehara-Wiener theorem on Dirichlet series with non-negative coefficients, which we recall next.

\begin{thm}[Ikehara-Wiener, \cite{MM}] \label{ikwe}
Let $A(s)= \sum_{n=1}^\infty a_n n^{-s}$ be a Dirichlet series. Suppose there exists another Dirichlet series $B(s)= \sum_{n=1}^\infty b_n n^{-s}$ with $b_n \ge 0$ such that

(a) $|a_n| \le b_n$,

(b) $B(s)$ converges in $\Re(s)>1$,

(c) $B(s)$ (respectively $A(s)$) can be extended meromorphically to $\Re(s) \ge 1$ having no poles except (resp. except possibly) for a simple pole at $s=1$ with residue $R \ge 0$ (resp. $r \ge 0$). Then,
\[ \sum_{n \le x} a_n = rx + o(x); \q (ii) \sum_{n \le x} b_n = Rx + o(x). \]
\end{thm}

\begin{proof}[Proof of \lemref{appln}]
Since $\mc F^1_M = \cup_\chi \mc F_{M,\chi}$ where $\chi$ runs over Dirichlet characters $\bmod M$ such that $m_\chi | M$, we can write, 
\begin{equation} \label{expan}
f(z)= \sum_\chi \sum_M \sum_\ell c_{\chi,M,\ell} \, f_{M}(\ell z),
\end{equation}
where $\chi$ runs over Dirichlet characters $\bmod N$, $M$ runs over divisors of $N$ such that $m_\chi \mid M$, $f_M$ runs over $\mc F_{M,\chi}$, $\ell$ runs over the divisors of $N/M$; and $c_{\chi,M,\ell} $ are scalars. By our assumption, there exist an $\chi,M$ such that $c_{\chi,M,1} \neq 0$.

From \eqref{expan}, it follows immediately that for primes $(p,N)=1$ we can write
\begin{equation}
a_f(p) = \sum_{\mk f} c_{\mk f} \lambda_{\mk f}(p)
\end{equation}
where $\mk f$ runs over the set $\cup_{\chi ,M} \mc F_{M, \chi}$, with $\chi$ and $M$ varying as above; and not all the $c_{\mk f}$ are zero. The point to note here is that all the $\mk f$ are newforms of level dividing $N$.

By a theorem of Shahidi \cite{Sha}, we know that the Rankin-Selberg convolution (in the sense of Langlands) $L(\mk f \otimes \overline{\mk g}, s) \ne 0$ on the line $\Re(s)=1$, if $\mk f \ne {\mk g}$. Moreover, it is classical (see e.g., \cite{Iwa}) that $L(\mk f \otimes \overline{\mk g}, s)$ is analytic in $\mf C$ except for a simple pole with positive residue at $s=1$ if and only if $\mk f = {\mk g}$.

Then we compute
\begin{equation}
\begin{aligned}
\sum_{p \le x, \, p \nmid N} |a_{\mk f}(p)|^2 \log p &= \sum_{\mk f} |c_{\mk f}|^2\sum_{p \le x, \, p \nmid N} |\lambda_{\mk f} (p)|^2 \log p + \sum_{\mk f \ne \mk g} c_{\mk f} \overline{c_{\mk g} }\sum_{p \le x, \, p \nmid N} \lambda_{\mk f} (p) \overline{\lambda_{\mk g} (p)}  \log p \n \\
&= \sum_{\mk f} |c_{\mk f}|^2\sum_{p \le x,\, p \nmid N} |\lambda_{\mk f \otimes \overline{\mk f}} (p)| \log p + \sum_{\mk f \ne \mk g} c_{\mk f} \overline{c_{\mk g} } \sum_{p \le x,\, p \nmid N} \lambda_{\mk f \otimes \overline{\mk g}} (p) \log p  \label{lincom}
\end{aligned}
\end{equation}

We would now appeal to (a relative version of) the Ikehara-Wiener theorem (see \thmref{ikwe}) applied to the Dirichlet series defined by the logarithmic derivatives 
$\frac{L'}{L}(\mk f \otimes \overline{\mk f},s)$ and then to $\frac{L'}{L}(\mk f \otimes \overline{\mk g},s)$. Let us denote the Satake parameters of $\mk f$ and $\mk g$ at a prime $p$ (which we suppress mostly) by $\{a_{i,p}\}$ and $\{b_{j,p}\}$ respectively. We drop the suffix $p$ from $\{a_{i,p}\}, \{b_{i,p}\}$ for convenience.

Further let us put
\begin{equation}
\frac{L'}{L}(\mk f \otimes \overline{\mk g},s) = \sum_{n=1}^\infty \Lambda_{\mk f \otimes \overline{\mk g}} (n) n^{-s};
\end{equation}
where
\[  \Lambda_{\mk f \otimes \overline{\mk g}} (n) = \begin{cases} (\sum_{i,j} a_i^m \overline{b_j}^m) \log p, \text{  if  } n = p^m \\ 0 \text{  otherwise}.    \end{cases}   \]

Thus by the Cauchy-Schwarz inequality, we see that
\begin{equation}
|\Lambda_{\mk f \otimes \overline{\mk g}} (n)| \le \frac{1}{2} \left(  \Lambda_{\mk f \otimes \overline{\mk f}} (n) + \Lambda_{\mk g \otimes \overline{\mk g}} (n)   \right),
\end{equation}
and moreover $\Lambda_{\mk f \otimes \overline{\mk f}} (n) = | \sum_i a_i^m|^2 \ge 0$ for all $n$. 

Thus all the conditions of the Ikehara-Wiener theorem (\thmref{ikwe}) for $\frac{L'}{L}(\mk f \otimes \overline{\mk f},s)$ and $\frac{L'}{L}(\mk f \otimes \overline{\mk g},s)$ are satisfied, and we have
\begin{equation} \label{asympIK}
(i) \,  \sum_{n \le x} \Lambda_{\mk f \otimes \overline{\mk f}} (n) = x + o(x), \q (ii) \, \sum_{n \le x} \Lambda_{\mk f \otimes \overline{\mk g}} (n) =  o(x)
\end{equation}
as $x \to \infty$.

It is then easy to finish the proof by noting that
\begin{equation} \label{pdom}
\sum_{n \le x} \Lambda_{\mk f \otimes \overline{\mk g}} (n) = \sum_{p \le x} \Lambda_{\mk f \otimes \overline{\mk g}} (p) +O(x^{1/2} \log x)
\end{equation}
and $\Lambda_{\mk f \otimes \overline{\mk g}} (p) = \lambda_{\mk f \otimes \overline{\mk g}} (p) \log p$. Indeed combining \eqref{lincom}, \eqref{asympIK} and \eqref{pdom} we get
\begin{equation} \label{asympgen}
\sum_{p \le x, (p,N)=1} |a_{\mk f}(p)|^2 \log p= \sum_{\mk f} |c_{\mk f}|^2 x + o_{\mk f}(x) ,
\end{equation}
as $x \to \infty$, and where $\sum_{\mk f} |c_{\mk f}|^2 >0$.
This immediately implies the assertion of \lemref{appln}.
\end{proof}

\begin{rmk} One may look for a better error term in \eqref{asympgen} from the point of view of analytic number theory. This maybe obtained (using the same arguments as above) by using the prime number theorem for the Rankin-Selberg $L$-functions $L(\mk f \otimes \overline{\mk g},s)$. However one has to be careful about Siegel zeros in case $L(\mk f \otimes \overline{\mk g},s)$ has a quadratic Dirichlet $L$-function as a factor.
\end{rmk}

\subsection{Proof of \thmref{oddthm}} First we note that in the statement of \propref{parta}, we can replace the condition: ``\textsl{odd and square-free}" in the second part of the proposition ``\textsl{by odd and prime}".

We appeal to \corref{nton1} (which is now unconditional, as we have proved \thmref{mainthm}) to conclude that our $F \in S^n_\rho$ has infinitely many non-zero Fourier-Jacobi coefficients $\phi_T$ ($T \in \Lambda^+_{n-1}$) with $d_T$ odd and square-free. Then \propref{parta} gives us the non-zero cusp form $H_\mu$ ($\mu$ primitive) as defined just before \propref{parta}. Let us keep the notation used there.

From now on we assume that $n$ is \textsl{odd}. Now let us observe that \lemref{appln} can be applied to $H_\mu \in S_{k'-\frac{n-1}{2}}(\Gamma_1(d'^2))$ (recall that $d' = d_T$ in the present case and $k'-\frac{n-1}{2} \in \mf N$) because its Fourier expansion is supported on indices which are co-prime to $d_T$ by the primitiveness of $\mu$. This in turn implies by old-form theory that it can not entirely lie in the old-space (see \cite[Theorem~4.6.8]{miyake}), and hence must have a non-zero new component.

Thus by the first paragraph of this section, we conclude that when $n$ is odd, there exist infinitely many odd primes $p$ such that for each such $p$,
$F$ has at least one non-zero Fourier coefficient $a_F(\mc T)$ with $d_{\mc T}=p$. The quantitative version follows immediately from the corresponding statement of \lemref{appln}. {\qed}

\begin{rmk}
It is desirable to prove an analogue of \thmref{oddthm} for $n$ even; however, the necessary properties of elliptic modular forms of half-integral weight (i.e., a suitable version of \lemref{appln}) seems not to be available.
\end{rmk}

\section{Appendix} \label{appn}
The reader may have noticed that while treating the analytic part, we have used two methods. As mentioned in the introduction, both have their advantages and limitations. The purpose of the appendix is to lay down the first steps of an asymptotic analysis by the extended Rankin-Selberg method which is adapted to treat non-cusp forms for both integral and half-integral weights. We show that the residue at the \textsl{rightmost} pole of Rankin-Selberg $L$-series of two not necessarily cusp forms defines an inner product on the space of Eisenstein series (!), see Definition~\ref{innp}. This is new to our knowledge; and has the potential to generalise to higher degrees and may have other applications.
Perhaps equally interesting is the intermediate result about detecting cuspidality from the growth of Fourier coefficients of modular forms of half-integral weights on any congruence subgroup, see \lemref{half-hecke}.

\textsl{We let $\kappa \in \tfrac{1}{2} \mf Z$ unless stated otherwise}. 

We aim to be a completely general in our statements, and so we start with $\mk f \in M_{\kappa}(\Gamma(N'))$ ($N'$ as in \eqref{n'}). By the theory of invariant differential operators and the Rankin-Selberg method, for $\mk f, \mk g$ as above, one can obtain a meromorphic continuation of the Rankin-Selberg convolution series (defined initially for $\Re(s) \gg 1$)
\begin{displaymath} \label{rfg}
 \mc R(\mk f, \mk g;s) = \sum_{n \geq 1} a_{\mk f}(n) \overline{a_{\mk g}(n) } n^{-s}.
\end{displaymath}
In particular, one obtains by the usual unfolding method the following integral representation (see eg. \cite[(3.5)]{Bo-Ch})
\begin{align}
\int_{\mc F_{N'}} E_{N'}(\tau, s+ 1 - \kappa)  R(\Im(\tau)^\kappa \mk f(\tau)  \overline{\mk g(\tau)}  ) d^*\tau
= \frac{2^{1-s} }{\pi^{ns} {N'}^{s+2} } s (s+1-2 \kappa) \, \Gamma(s)  \mc R (\mk f,\mk g;s) , \label{int}
\end{align}
where $E_{N'}$ is the real analytic Eisenstein series at $\infty$ for the group $\Gamma(N')$ (see e.g., \cite{Iw1}, \cite[(5.1.8)]{Ran}), $R$ denotes a `growth-killing' $\mrm{SL}(2, \mf R)$-invariant differential operator defined on two real analytic functions $\mf f$ and $\mf g$ by
\begin{displaymath} \label{Rdiffop}
{R}(v^{\kappa} \mf f \cdot\mf g)= 4 \mf f'\cdot \bar{ \mf g'}v^{\kappa+2} +
2i \kappa(\mf f'\cdot\bar{\mf g}- \mf f \cdot \bar{\mf g'})v^{\kappa+1} \qq \qq (v = \Im(\tau))
\end{displaymath}
and $d^*\tau$ denotes the invariant measure on $\mf H$.

From \cite{Iw1} we note that the only pole of $E_{N'}$ is at $s=1$ and is simple with residue a constant function of $\tau$. Now \eqref{int} suggests that $\mc R(\mk f, \mk f;s)$ has a (simple) pole at $s=2\kappa-1$. We confirm this in \propref{2k-1} below.

The assertion about the pole is not clear immediately from the integral representation, and we take an indirect approach to prove it. The following lemma comes in handy for that. It seems this lemma is not available in the literature in the generality that we want, especially when the weight is half-integral, see e.g., \cite{ch-ko}.

\begin{lem} \label{half-hecke}
Let $N \ge 1, \kappa \geq 3/2$ and $\mk f \in \mN$ be non-zero. If $|a_{\mk f}(n)| \ll_{\mk f} n^{c}$ for any fixed $c < \kappa-1$. Then $\mk f$ must be a cusp form.
\end{lem}

\begin{proof}
If $\kappa \in \mf N$, then this follows directly from \cite[Theorem~1.2]{BD1}. So assume $\kappa = k+1/2 \in \frac{1}{2} \mf Z \setminus \mf Z$ and let $\mk f$ be as in the statement of the lemma. Consider the modular form (recall that $N'=4N$ in this case)
\[  \mk g(\tau):= \mk f(4N \tau) \cdot \theta(\tau) \in M_{k+1}(\Gamma_1(16N^2)),  \]
where $\theta(\tau) = \sum_{n \in \mf Z} q^{n^2}$. Thus if $\mk f(\tau) = \sum_{n \geq 1} a_{\mk f}(n)q^{n/4N}$, then from the definition of $\mk g$,
\[ a_{\mk g}(n) = \sum_{ \ell +m^2=n} a_{\mk f}(\ell) . \]
Since $\{ (\ell,m) | \ell +m^2=n \} \ll n^{1/2}$, by using the bound $|a_{\mk f}(n)| \ll_f n^{c}$ we easily get that
\[  |a_{\mk g}(n)| \ll_{\mk g} n^{c+1/2}. \]

That $\mk g$ is a cusp form now follows again from \cite[Theorem~1.2]{BD1} noting that $k+1 \geq 2$. We claim  this implies that ${\mk f}_{4N} (\tau):= \mk f(4N \tau) \in M_{k+1/2}(\Gamma_1(16N^2))$ is also a cusp form. It is enough to show that the order of vanishing $\mrm{ord}_{s} ({\mk f}_{4N})$ at any cusp $s$ of $\Gamma_1(16N^2)$ is positive.

To see this, first note that at the cusps $\infty, 0, 1/2$ of $\Gamma_0(4)$,
the theta function looks like $\sum_{n \in \mf Z} q^{n^2}$, $\sum_{n \in \mf Z} q^{n^2/4}$, $q^{1/4} \sum_{n \in \mf Z} q^{n^2+n}$ respectively, possibly up to some non-zero constants. 
Since any cusp of $\Gamma_1(16N^2)$ is $\Gamma_0(4)$ equivalent to one of the cusps of $\Gamma_0(4)$, and $\theta$ is automorphic under $\Gamma_0(4)$, we can write
\[ 1 \leq \mrm{ord}_{s} (\mk g) =  \mrm{ord}_{s} ({\mk f}_{4N})  + \mrm{ord}_{l(s)} (\theta) ,\]
where $l(s) \in \{ \infty, 0, 1/2 \}$. When $l(s)=\infty, 0$, we see that $\mrm{ord}_{l(s)} (\theta)=0$
and the                                                                                                                                                                                                                                                                                                                                                                                                                                                                                                                                                                                                                                                                                                                                                                                                                                                                                                                                                                                                                                                                                                                                                                                                                                                                                                                                                                                                                                                                                                                                                                                                                                                                                                                                                                                                                                                                                                                                                                                                                                                                                                                                                                                                                                                                                                                                                                                                                                                                                                                                                                                   assertion about $\mk f_{4N}$ is clear.  When $l(s) =1/2$, we see that 
\[    \mrm{ord}_{s} ({\mk f}_{4N})  \ge 1- \mrm{ord}_{1/2} (\theta) =3/4.\]
This implies that $\mk f_{4N}$ is cuspidal and hence so is $\mk f$. 
\end{proof}

\begin{prop} \label{2k-1}
Let $\mk f \in \mN$ be non-zero. Assume that $\mk f$ is not a cusp form and that $\kappa \ge 3/2$. Then the Rankin-Selberg series $\mc R(\mk f,\mk f;s)$ has a simple pole with positive residue at $s=2 \kappa-1$ and 
\begin{equation} \label{asymp}
\sum_{n \leq x} |a_{\mk f}(n)|^2 = c'_{\mk f} x^{2\kappa-1} +O_{\mk f}(x^\eta) 
\end{equation}
for some constant $c'_{\mk f}>0$ not depending on $x$.
\end{prop}

\begin{proof}
First of all from \eqref{int}, it is clear that $\mc R(\mk f,\mk f;s)$ has meromorphic continuation to $\mf C$ with the exception of possible simple poles at $s=\kappa,2\kappa-1$. We claim that there must be a pole at $s=2 \kappa-1$. If not, then by Landau's theorem on Dirichlet series with non-negative coefficients, the Dirichlet series representing $\mc R(\mk f,\mk f;s)$ would actually converge for all $\sigma > \kappa$. Thus $|a_{\mk g}(n)|= O(n^{c})$ for any $c$ such that $\kappa/2<c$. In particular considering such a $c$ satisfying $c< \kappa-1$ violates \lemref{half-hecke}. 

Therefore by a version of Landau's theorem we would get an asymptotic formula of the kind provided we knew a `generalized' functional equation for $\mc R^*(\mk f,\mk f;s)$, where $\mc R^*(\mk f,\mk f;s)$
denotes the completion of $\mc R(\mk f,\mk f;s)$ defined by (cf. \cite[Rmk.~3.3]{Bo-Ch} taking $l=0$ there, also note that $\kappa \in \mf N$ there, but the same argument works in our case, see below)
\begin{equation*}
{\mc R}^*(\mk f,\mk f;s) = 2^{1-s} \pi^{-s} \Gamma(s) \xi(2s+2-2\kappa) {\mc R}(\mk f,\mk f;s) .
\end{equation*}
Here $\xi(s) = \pi^{-s/2} \Gamma(s/2) \zeta(s)$. If we further put 
\begin{equation} \label{phidef}
 \phi(s):= \zeta(2s+2-2\kappa)  {\mc R}(\mk f,\mk f;s) = \sum_{n \ge 1} b_n n^{-s},  
\end{equation}
then clearly $\phi(s)$ has non-negative Dirichlet coefficients and 
\[  {\mc R}^*(\mk f,\mk f;s) = 2^{1-s} \pi^{1-\kappa}  \Gamma(s) \Gamma(s+1-\kappa) \phi(s). \]
$\phi(s)$ satisfies a functional equation relating ${\mc R}^*(\mk f,\mk f; 2 \kappa-1- s)z$ with a linear combination of similar objects arising from Rankin-Selberg series of $\mk f$ at different cusps. See \cite[Thm.~4~part~(iv)]{Ran1} for the case of cusp forms, and the comments below.

We now invoke the asymptotic result from Chandrasekharan-Narashimhan \cite[Theorem~4.1]{chandra} to get
\begin{equation} \label{bn}
\sum_{n \le x} b_n =  \sum_{\text{all poles}} \mathrm{Res} \left( \frac{\phi(s) x^s}{s} \right) +O_{\mk f}(x^\eta)  ,
\end{equation}
where `all poles' means contributions from all the poles of $\phi(s)$, and $\eta<2 \kappa-1$.
Note that in the notation of \cite{chandra}, we have the central point of the functional equation to be $ \kappa-1/2$ and the rightmost abscissa of a pole of $\phi(s)$ to be $2\kappa-1$.

Indeed by Selberg's remarks (see \cite[comment after~(1.15)]{Selb}) and \cite[Remark~B.]{Ran1}, their method would apply to the function $R(\Im(\tau)^\kappa |\mk f(\tau)|^2)$ since it has exponential decay and it boils down to the functional equation of $E_{N'}$ (which comes from the scattering matrix for $\Gamma(N')$); which fortunately has been worked out by Rankin (see \cite[Thm.~4~part~(iv)]{Ran1}). This gives the analytic continuation and functional equation of $\mc R^*(\mk f,\mk f;s)$ via \eqref{int} (see also \cite{Bo-Ch}). Finally in order to use the result from \cite{chandra}, we require that $\phi(s)$ has finitely many poles: this is clear. For the asymptotic property, one could also consult \cite[Section~5.3]{Ran1} and the treatment in our case would be the same.

Clearly the `leading term' in \eqref{bn} is of the form $c_{\mk f} x^{2\kappa-1}$ and $c_{\mk f}$ is necessarily non-zero since there is a simple pole at $s=2\kappa-1$. It must be positive, since the L.H.S. of \eqref{bn} is strictly positive for large $x$. Finally it is standard (see the calculation on \cite[p.364-365]{Ran1}) that one can deduce the following asymptotic (for some $c'_{\mk f}>0$) for the second moment of the Fourier coefficients of $\mk f$ using the relation between $b_n$ and $a_{\mk f}(n)$ from \eqref{phidef}. This gives us
\begin{equation*} 
\sum_{n \leq x} |a_{\mk f}(n)|^2 = c'_{\mk f} \cdot x^{2\kappa-1} +O_{\mk f}(x^\eta) 
\end{equation*} 
as desired. The proposition is now proved.
\end{proof}

\begin{defi} \label{innp}
For $\mk f,\mk g \in \mN$, let us define a pairing $\{\mk f,\mk g\}$ by putting
\begin{displaymath}
\{\mk f, \mk g\} := \mrm{res}_{s=2\kappa-1} \mc R(\mk f, \mk g,s),
\end{displaymath}
\end{defi}
\noindent where we again denote by $\mc R(\mk f, \mk g,s)$ the analytic continuation (via \eqref{int}) of the Dirichlet series defined in \eqref{rfg}. Clearly the pairing $\{\mk f,\mk g\}$ is Hermitian and by \propref{2k-1}, is positive definite on the space $\mc E_\kappa(N')$ of Eisenstein series, which we can define as the orthogonal complement of the space of cusp forms in $\mN$. Thus it defines an inner product on $\mc E_\kappa(N')$. In fact $\{ \mk f, \mk f\}>0$ for any non-cusp form.

Moreover let us note that if $\{\mk f, \mk g \}=0$ if at least one of $\mk f, \mk g$ is cuspidal, this follows (see eg. \cite[p.~804]{Bo-Ch}) from the fact that the associated Rankin-Selberg $L$-series atmost a simple pole at $s=k$. In the passing let us also mention that $\{\mk f, \mk g \}$ should not be confused with the extended bilinear pairing $(\mk f, \mk g )=\mrm{res}_{s=\kappa} \mc R(\mk f, \mk g,s)$. It is well-known that $(\mk f, \mk g )$ must have a non-trivial signature in general (see eg. \cite[Rmk.~5.7]{Bo-Ch}).

The inner product $\{ \cdot , \cdot \}$ is not $\sltwor$-invariant, unfortunately. Nor is the classical Hecke-basis (as in \cite{weis}) orthogonal with respect to it (cf. \cite[p.~820]{Bo-Ch}). However for each $m \ge 1$, there must exist an element $\mk E_m \in \mc E_\kappa(N')$ which is dual to the linear functional $\mk f \mapsto a_{\mk f}(m)$ -- with properties analogous those of the classical Poincar\'{e} series in the space of cusp forms. In particular it may enable one to prove an analogue of the asymptotic result in \cite{DI}, which would then be sufficient to treat the case $n$ even in \thmref{mainthm}. It will be interesting to investigate this further.


\begin{thebibliography}{99}

\bibitem{abra} N. Abramowitz, I. Stegun, {\em Handbook of Mathematical Functions}, Dover, New York, 1965.

\bibitem{AJ} A. Ajouz, {\em }. Hecke Operators on Jacobi Forms of Lattice Index and the Relation to Elliptic Modular Forms, Ph.D. thesis, University of Siegen, 2015. \url{https://dspace.ub.uni-siegen.de/bitstream/ubsi/938/1/Dissertation_Ali_Ajouz.pdf}

\bibitem{AD} P. Anamby, S. Das: {\em Distinguishing Hermitian cusp forms of degree $2$ by a certain subset of all Fourier coefficients}. Publ. Mat. 63 (2019), 307-341.

\bibitem{An3} A. N. Andrianov: {\em Euler products corresponding to Siegel modular forms of genus $2$}. Russian 
Math. Surveys 29:3 (1974), 45-116.


\bibitem{An1} A. N. Andrianov: {\em Euler decompositions of theta-transformations of Siegel modular forms}. Math. USSR-Sb. 34(1978), 291-341.

\bibitem{An2} A. N. Andrianov: {\em The multiplicative arithmetic of Siegel modular foms}. Russian 
Math. Surveys 34(1979), 75-148.

\bibitem{AS} M. Asgari, R. Schmidt: {\em Siegel modular forms and representations}. manuscripta math. 104, 173- 200 (2001).

\bibitem{bo} A. Balog, K. Ono: {\em The Chebotarev density theorem in short intervals and some questions of Serre}. J. Number Theory., 91(2), 2001, 356-371.

\bibitem{Bo} S. B\"ocherer: {\em \"{U}ber die Fourier-Jacobi-Entwicklung Siegelscher Eisensteinreihen}. Math. Z., \x{183}, 21--46 (1983).

\bibitem{Bo-Ch} S. B\"ocherer, F. L. Chiera: {\em On Dirichlet series and Petersson products of Siegel modular forms}. Annales de L'institut Fourier, Tome \x{8}, no$^\circ$~3, (2008), 801--824.


\bibitem{BD1}
S. B\"ocherer, S. Das: {\em Cuspidality and the growth of Fourier coefficients of modular forms}. J. Reine Angew. Math., Volume 2018, Issue 741, Pages 161-178.

\bibitem{BD2} S. B\"ocherer, S. Das: {\em Petersson norms of not necessarily cuspidal Jacobi modular forms and applications}. Adv. Math., 336, (2018), 336-375. 


\bibitem{BD3} S. B\"ocherer, S. Das: {\em On Fourier coefficients of elliptic modular forms $\bmod \ell$ with applications to Siegel modular forms}. manuscripta math., to appear \url{https://arxiv.org/abs/2001.11700}.

\bibitem{bo-ko}
S. B\"ocherer, W. Kohnen: {\em On the Fourier coefficients of Siegel Modular Forms}. Nagoya Math. J., Volume 234, June 2019, pp. 1-16.

\bibitem{borel}
A. Borel: {\em Linear Algebraic groups}. Second edition. Graduate Texts in Mathematics, 126. Springer-Verlag, New York, 1991. xii+288 pp.



\bibitem{Byko}
V. A. Bykovski\u{i}: {\em A trace formula for the scalar product of Hecke series and its applications.}
Zap. Nauchn. Sem. S.-Peterburg. Otdel. Mat. Inst. Steklov. (POMI), 226(Anal. Teor.
Chisel i Teor. Funktsii. 13):14-36, 235-236, 1996.


\bibitem{cas} J. W. S. Cassels: {\em Rational Quadratic forms}. Academic Press, 1978.

\bibitem{chandra} K. Chandrasekharan, R. Narasimhan, {\em Functional equations with multiple gamma factors and the average order of arithmetical functions.} Ann. of Math. (2) 76 (1962), 93-136.

\bibitem{ch-ko} Y. Choie, W. Kohnen: {\em On the Fourier coefficients of modular forms of half-integral weight}. Int. J. Number Theory 9 (2013), no. 8, 1879-1883.

\bibitem{Cou-Pan} M. Courtieu, A. Panchishkin: {\em Non-archimedean L-functions and arithmetical
modular forms}. Lecture Notes in Mathematics 1471, Springer Berlin-Heidelberg, 1991, 2004.

\bibitem{D} H. Delange: {\em G\'{e}n\'{e}ralisation du th\'{e}or\`{e}me de Ikehara}. Ann. scient. Ec. Norm. Sup. 9 S\'{e}rie 3, 71 (1954), 213-242.

\bibitem{DI} W. Duke, H. Iwaniec: {\em Bilinear forms in the Fourier coefficients
of half-integral weight cusp forms and sums over primes}.  Math. Ann. 286 (1990), 783-802 .

\bibitem{EZ} M.\ Eichler and D.\ Zagier: {\em The Theory of Jacobi Forms}. Progress in Mathematics, \x{Vol. 55},  Boston-Basel-Stuttgart: Birkh\"auser, 1985.

\bibitem{evdo} S. A. Evdokimov: {\em Dirichlet series, multiple Andrianov zeta-functions in the theory of Siegel modular forms of genus $3$}, Dokl. Akad. Nauk SSSR 277 (1984), 25-29.

\bibitem{Fr} E. Freitag: {\em Siegelsche Modulfunktionen}. Grundl. Math. Wiss., \x{254} Springer--Verlag, (1983).

\bibitem{Fr2} E. Freitag: {\em Ein Verschwindungssatz f\"ur automorphe Formen zur 
Siegelschen Modulgruppe}. Math. Z. \x{165}, (1979), 11-18.


\bibitem{Ibu} T. Ibukiyama, R. Kyomura: A generalization of vector-valued Jacobi forms. Osaka J.Math. 48, (2011) 783-808.

\bibitem{Iwa} H. Iwaniec: {\em Topics in classical automorphic forms},
(Graduate Studies in Mathematics, {17}, American Mathematical Society, Providence, RI 1997).

\bibitem{Iw1} H. Iwaniec: {\em Introduction to the spectral theory of automorphic forms.} Revista Matem\'{a}tica Iberoamericana, Madrid, 1995. xiv+247 pp.

\bibitem{Ki1}  Y. Kitaoka, {\em Modular forms of degree $n$ and representations by Quadratic forms}. Nagoya Math. J., vol 74 1979 95--122.

\bibitem{Ki} Y. Kitaoka: {\em Arithmetic of quadratic forms.} Cambridge Tracts in Mathematics, 106. Cambridge University Press, Cambridge, 1993. x+268 pp.

\bibitem{Kl} H. Klingen: {\em Introductory lectures on Siegel modular forms}, Cambridge Studies in Advanced Mathematics, {20}, Cambridge University Press, (1990).



\bibitem{Lang} S. Lang: {\em Introduction to Modular Forms} (With appendixes by D. Zagier and Walter Feit., Corrected reprint of the 1976 original). Grundlehren der Mathematischen Wissenschaften, 222. Springer-Verlag, Berlin, 1995. x+261 pp. 

\bibitem{yli} Y. Li: {\em Restriction of Coherent Hilbert Eisenstein series}. Math. Ann., Volume 368, Issue 1-2, pp 317-338.

\bibitem{miyake} T. Miyake: {\em Modular forms}. Translated from the 1976 Japanese original by Yoshitaka Maeda. Reprint of the first 1989 English edition. Springer Monographs in Mathematics. Springer-Verlag, Berlin, 2006. x+335 pp.

\bibitem{MM} M. Ram. Murty, V. Kumar Murty: {\em Non-vanishing of $L$-functions and applications}. Progress in Mathematics, 157. Birkh\"auser Verlag, Basel, 1997. xii+196 pp.

\bibitem{Pet1} H. Petersson: {\em 
\"{U}ber die Entwicklungskoeffizienten der ganzen Modulformen und ihre Bedeutung f\"{u}r die Zahlentheorie}.
Hamburg Abh., 8, (1931), 215-242.

\bibitem{Pet2} H. Petersson: {\em \"{U}ber die systematische Bedeutung der Eisensteinschen Reihen.} (German) Abh. Math. Sem. Univ. Hamburg 16, nos. 1-2, (1949) 104-126.

\bibitem{sps} A. Pitale, A. Saha, R. Schmidt: {\em Transfer of Siegel cusp forms of degree $2$}. Mem. Amer. Math. Soc. 232 (2014), no. 1090, vi+107 pp.

\bibitem{Pol} A. Pollack: {\em The spin $L$-function on $GSp_6$ for Siegel modular forms}. Compos. Math. 153 (2017), no. 7, 1391-1432.

\bibitem{Ran} R. A. Rankin: {\em Modular forms and functions}. Cambridge University Press, Cambridge-New York-Melbourne, 1977. xiii+384 pp.

\bibitem{Ran1} R. A. Rankin: {\em Contributions to the theory of Ramanujan's function $\tau(n)$ and similar arithmetic functions, II Order of the Fourier coefficients of integral Modular Forms}. Proc. Cambridge Philos. Soc., 35 (1939), 357-372.

\bibitem{saha} A. Saha: {\em Siegel cusp forms of degree 2 are determined by their fundamental Fourier coefficients}. Math. Ann. 355 (2013), no. 1, 363-380.

\bibitem{Scho} B. Schoeneberg: {\em Elliptic modular functions: an introduction.} (Translated from the German by J. R. Smart and E. A. Schwandt). Die Grundlehren der mathematischen Wissenschaften, Band 203. Springer-Verlag, New York-Heidelberg, 1974. viii+233 pp.



\bibitem{Selb} A. Selberg: {\em On the estimation of Fourier coefficients of modular forms.} 1965 Proc. Sympos. Pure Math., Vol. VIII pp. 1-15 Amer. Math. Soc., Providence, R.I. 

\bibitem{Se-St} J-P. Serre, H. Stark: {\em Modular forms of weight 1/2}. Modular functions of one variable, VI (Proc. Second Internat. Conf., Univ. Bonn, Bonn, 1976), pp. 27-67. Lecture Notes in Math., Vol. 627, Springer, Berlin, 1977.

\bibitem{serre1} J-P. Serre: {\em Quelques applications du th\'{e}or\`{m}e de densit\'{e} de Chebotarev}. (French) [Some applications of the Chebotarev density theorem] Inst. Hautes \"{E}tudes Sci. Publ. Math. No. 54 (1981), 323-401.

\bibitem{serre2} J-P. Serre: {\em Divisibilit\'{e} de certaines fonctions arithm\'{e}tiques.} S\'{e}minaire Delange-Pisot-Poitou. Th\'{e}orie des nombres, tome 16, no$^\circ$1 (1974-1975), exp. no$^{\circ}$ 20, p. 1-28.

\bibitem{Sha} Freydoon Shahidi: {\em On Nonvanishing of $L$-Functions}. Bulletin of the American Mathematical Society 2:3 (1980), 462-464.



\bibitem{Sko} N.-P. Skoruppa: {\em \"Uber den Zusammenhang zwischen Jacobiformen
und Modulformen halbganzen Gewichts}. Ph. D.Thesis, Bonn 1985.

\bibitem{SZ} N.-P. Skoruppa, D. Zagier: {\em Jacobi forms and a certain space of modular forms}. Invent. Math., 94 (1988), 113-146.

\bibitem{Str} F. Stromberg: Weil representations associated with finite quadratic modules. Math. Z. (2013) 275, 509-527.

\bibitem{voi} J. Voight: {\em Quaternion Algebras.} \url{https://math.dartmouth.edu/~jvoight/quat-book.pdf}.

\bibitem{weis} J. Weisinger: {\em Some results on classical Eisenstein series and modular forms over function fields}. Thesis, Harvard Univ. (1977).


\bibitem{Yam} S. Yamana: {\em Determination of holomorphic modular forms by
primitive Fourier coefficients}. Math. Ann. 344, (2009), 853-863.



\bibitem{Zi} C. Ziegler: Jacobi forms of higher degree. Abh.Math.Sem. 
Univ.Hamburg 59, 191-224 (1989)




\end{thebibliography}
\end{document}